\documentclass{article}
\usepackage[utf8]{inputenc}
\usepackage{comment}
\usepackage{caption}
\usepackage{amsmath}
\usepackage{amsthm}
\usepackage{amsfonts}
\usepackage{amssymb}
\usepackage{xcolor}
\usepackage{wasysym}
\usepackage{bbm}
\usepackage{float}
\usepackage{verbatim}
\usepackage{hyperref}
\usepackage{listings}
\usepackage{tablefootnote}
\usepackage{graphicx}
\providecommand{\keywords}[1]
{
  \small	
  \textbf{\textit{Keywords---}} #1
}
\hypersetup{
  colorlinks   = true,    
  urlcolor     = blue,    
  linkcolor    = blue,    
  citecolor    = red      
}

\newcommand{\dist}{\textrm{d}}

\newcommand{\odd}{\textrm{odd}}
\newcommand{\even}{\textrm{even}}
\newcommand{\Perm}{\textrm{Perm}}
\newcommand{\argmax}{\textrm{argmax}}

\title{ 
On the Wasserstein Distance Between $k$-Step Probability Measures on Finite Graphs}
\author{Sophia Benjamin\footnote{North Carolina School of Science and Mathematics, Durham, NC. \, \, \,\,\,\,\,\,\,\,\,\,\,\,\,\,\,\,\,\,,\,\,\,\,\,\,\,\,\,\,\,\,\,\,\,\,\,\,\,\,\,\,\,\,\,\,\,\,\,\,\,\, Email: sophia.r.benjamin@gmail.com},  
Arushi Mantri\footnote{Jesuit High School, Portland, OR. Email: arushi.mantri@gmail.com}, Quinn Perian\footnote{Stanford Online High School, Palo Alto, CA. Email: quinn.perian@outlook.com}}
\date{October 19, 2021}
\theoremstyle{plain}
\newtheorem{theorem}{Theorem}[section]
\newtheorem{corollary}[theorem]{Corollary}
\newtheorem{lemma}[theorem]{Lemma}

\newtheorem{remark}[theorem]{Remark}

\theoremstyle{definition}
\newtheorem{definition}[theorem]{Definition}
\newtheorem{example}[theorem]{Example}

\begin{document}
\maketitle
\begin{abstract}
   We consider random walks $X,Y$ on a finite graph $G$ with respective lazinesses $\alpha, \beta \in [0,1]$. Let $\mu_k$ and $\nu_k$ be the $k$-step transition probability measures of $X$ and $Y$. In this paper, we study the Wasserstein distance between $\mu_k$ and $\nu_k$ for general $k$. We consider the sequence formed by the Wasserstein distance at odd values of $k$ and the sequence formed by the Wasserstein distance at even values of $k$. We first establish that these sequences always converge, and then we characterize the possible values for the sequences to converge to. We further show that each of these sequences is either eventually constant or converges at an exponential rate. By analyzing the cases of different convergence values separately, we are able to partially characterize when the Wasserstein distance is constant for sufficiently large $k$.
\end{abstract}


\section{Introduction}

Optimal transport theory concerns the minimum cost, called the \textit{transportation distance}, of moving mass from one configuration to another. In this paper, the notion of transportation distance that we are concerned with is the $L^1$ transportation distance, which we refer to as the \textit{Wasserstein distance}. The Wasserstein distance has applications in fields such as image processing, where a goal is to efficiently transform one image into another (e.g., \cite{rubner2000earth}), and machine learning, where a goal is to minimize some transport-related cost (e.g., \cite{frogner2015learning}). 

The application of Wasserstein distance that motivates this paper is the definition of \textit{$\alpha$-Ricci curvature} $\kappa_\alpha$ on graphs introduced by Lin, Lu, and Yau in \cite{lin2011ricci}:
$$\kappa_\alpha = 1 - \frac{W(m_x^\alpha,m_y^\alpha)}{\dist(x,y)}.$$ 
Here $\dist(x,y)$ is the graph distance between vertices $x$ and $y$, while $m_v^\alpha$ is the 1-step transition probability measure of a random walk starting at vertex $v$ with laziness $\alpha$, and $W(m_x^\alpha,m_y^\alpha)$ is the Wasserstein distance between $m_x^\alpha$ and $m_y^\alpha$.

The $\alpha$-Ricci curvature is a generalization of classical Ricci curvature, an object from Riemannian geometry that captures how volumes change as they flow along geodesics (\cite{ollivier2011visual}). 
In \cite{ollivier2009ricci}, Ollivier created the Ollivier-Ricci curvature to generalize the idea of Ricci curvature to discrete spaces, such as graphs. The Ollivier-Ricci curvature between $X$ and $Y$ is defined via the Wasserstein distance between the $1$-step transition probability measures of random walks starting at $X$ and $Y$. It captures roughly whether the neighborhoods of $X$ and $Y$ are closer together than $X$ and $Y$ themselves. The Ollivier-Ricci curvature is well-studied in geometry and graph theory (\cite{jiradilok2021transportation}, \cite{CushingKamtue+2019+22+44}, \cite{bourne2018ollivier}, \cite{cushing2020rigidity}, \cite{van2021ollivier}), and is also used to study economic risk, cancer networks, and drug design, among other applications (\cite{sandhu2015graph}, \cite{sandhu2016ricci}, \cite{sia2019ollivier}, \cite{wang2016interference}, \cite{wee2021ollivier}, \cite{jiradilok2021transportation}). Lin, Lu, and Yau further generalized the Ollivier-Ricci curvature to $\alpha$-Ricci curvature (\cite{lin2011ricci}), allowing for the laziness $\alpha$ of the random walks considered to be greater than zero.

In \cite{ollivier2009ricci}, Ollivier suggested exploring Ollivier-Ricci curvature on graphs at ``larger and larger scales." Thus, in this paper, we study the Wasserstein distance between $k$-step probability measures of random walks with potentially nonzero laziness as $k$ gets larger and larger. Since $1$-step probability distributions of random walks were used to study the initial ``small-scale" $\alpha$-Ricci curvature, these $k$-step probability distributions are a natural way to understand curvature at ``larger and larger scales." Jiradilok and Kamtue (\cite{jiradilok2021transportation}) study these $k$-step distributions for larger and larger $k$ on infinite regular trees; in this paper, we study them instead on finite graphs.

Given a finite, connected, simple graph, we consider a random walk with starting vertex $w$ and laziness $\alpha$. The random walk is defined to be a Markov chain where at each step, we either stay at the current vertex with probability $\alpha$ or pick a neighboring vertex at random and move there. We then consider the probability distribution encoding the likelihood of being at each possible vertex after $k$ steps of this random walk, which is called a $k$-step probability distribution, or $k$-step probability measure. 

Given two such random walks on one graph, starting at vertices $u,v$ and with respective lazinesses $\alpha,\beta$, we define the Wasserstein distance between their two $k$-step probability measures to be the minimum cost of moving between the two distributions. Here, moving 1 unit of mass across 1 edge costs 1 unit. 

We can ask many questions about the Wasserstein distance at ``larger and larger scales." For instance, does the Wasserstein distance between the two $k$-step probability distributions always converge as $k \to \infty$? Also, what does it converge to in different cases? Even more interestingly, what can we say about the rate of convergence? In particular, when does the distance eventually remain constant, and how long could it take to reach constancy?

In this paper, we show in all cases that either the Wasserstein distance converges or the Wasserstein distance at every other step converges. We also classify what the distance converges to in all cases, addressing the first and second questions.

We then seek to understand the rate of convergence of the Wasserstein distance. We reach two main results. First, addressing the third question, we show that unless the Wasserstein distance at every other step is eventually constant, its rate of convergence is exponential (Theorem \ref{thm: Guvab Convergence Theorem}). We also address the fourth question by providing a partial characterization of exactly when the Wasserstein distance is eventually constant (Theorem \ref{thm: Characterization of Constancy}). 

In Section 2, we provide formal definitions of key concepts used throughout the paper. In particular, we recall the definition of the Wasserstein distance and introduce the notion of a Guvab. A \textit{Guvab} refers to a pair of random walks on a finite connected simple graph, and these Guvabs are the primary object we study in this paper. In Section 3, we classify for all possible Guvabs the limiting behavior of the Wasserstein distance, when the distance converges, and what the distance converges to. This characterization provides a natural way to classify the Guvabs into four categories based on their limiting behavior: $W=1$; $W=0$; $W=\frac{1}{2}$; and $\beta = 1$. In each of Sections 4, 5, 6, and 7, we consider one of these four categories of Guvabs and determine when the Wasserstein distance is eventually constant as well as examine the rate of convergence if the Wasserstein distance is not constant. Along the way, we encounter various interesting results about the different cases. Finally, in Section 8, we present main results about constancy and rate of convergence in general, obtained by considering each of these four cases individually.

\section{Preliminaries}

We begin with several formal definitions that we use in the remainder of the paper. We start by recalling graph theory terminology and the definition of Wasserstein distance on graphs. Then, we review random walks on graphs and define Guvabs. Finally, we briefly discuss terminology used to describe convergence.

In this paper, all graphs we consider are finite, connected, simple graphs. For a graph $G$, let $V(G)$ be the vertex set of $G$ and $E(G)$ be the edge set of $G$, i.e., the set of unordered pairs $\{v_1,v_2\}$ where $v_1,v_2$ are adjacent vertices in $G$. Further, for any $v\in V(G)$, let $N(v)$ be the neighbor set of $v$. Finally, denote by $\dist(w_1,w_2)$ the graph distance between vertices $w_1$ and $w_2$.

\begin{definition}
Define a \textbf{distribution} on the graph $G$ to be a function $\mu:V(G)\to \mathbb{R}$. We say $\mu$ is a \textbf{nonnegative distribution} if, for all $v\in V(G)$, we have $\mu(v)\geq 0$. A nonnegative distribution $\mu$ is a \textbf{probability distribution} if $\sum_{w \in V(G)} \mu(w) = 1$. 
\end{definition}

For convenience, we will denote by $\Tilde{\textbf{0}}$ the distribution with value $0$ at all vertices, (i.e., for all $v\in V(G)$, we have $\Tilde{\textbf{0}}(v) = 0$). In addition, we will refer to a distribution $\mu$ for which $\sum_{w \in V(G)} \mu(w) = 0$ as a \textbf{zero-sum distribution}.

Given a graph $G$, let $\{\mu_i\}_{i=0}^\infty$ be an infinite sequence of distributions. Suppose that $f: \mathbb{Z}_{\geq0} \to \mathbb{Z}_{\geq0}$ is a strictly increasing function such that for all vertices $w\in G$, $\displaystyle\lim_{k\to\infty}\mu_{f(k)}(w)$ exists. Then denote by $\displaystyle\lim_{k\to\infty}\mu_{f(k)}$ the pointwise limit. Namely, for all $w \in V(G)$, let $\displaystyle \left( \lim_{k\to\infty}\mu_{f(k)} \right)(w)$ be $\displaystyle\lim_{k\to\infty}(\mu_{f(k)}(w))$.

For a given graph $G$, let $D = D(G)$
be the set of all ordered pairs $(\mu,\nu)$ of distributions on $V(G)$ that satisfy $\sum_{w \in V(G)} \mu(w) = \sum_{w \in V(G)} \nu(w)$. Further, let $D_{\geq0}$ be the set of all ordered pairs $(\mu,\nu) \in D$ with $\mu,\nu$ nonnegative distributions.

We now introduce some terminology from optimal transport theory. We follow definitions equivalent to those in the book of Peyre and Cuturi \cite{COTFNT}. 

In Definitions~\ref{def: terry the transportation plan fairy},~\ref{def: carrie the cost function fairy}, and~\ref{def: warry the wasserstein distance fairy}, we let $G$ be a graph with two nonnegative distributions $\mu, \nu$ on $V(G)$ such that $(\mu,\nu)\in D_{\geq0}$.

\begin{definition}[c.f. \cite{COTFNT}]\label{def: terry the transportation plan fairy}
Define a \textbf{transportation plan} from $\mu$ to $\nu$ for $(\mu,\nu)\in D_{\geq0}$ to be a function $T_{\mu,\nu}: V(G)\times V(G) \to \mathbb{R}$ such that 
\begin{itemize}
    \item for any vertices $w_1, w_2 \in V(G)$, we have that $T_{\mu,\nu}(w_1,w_2) \geq 0$,
    \item for all vertices $w\in V(G)$, we have that $\sum_{i \in V(G)} T_{\mu,\nu}(w,i) = \mu(w)$,
    \item for all vertices $w\in V(G)$, we have that $\sum_{i \in V(G)} T_{\mu,\nu}(i,w) = \nu(w)$.
\end{itemize} 
Denote by $\mathcal{T}_{\mu,\nu}$ the set of all transportation plans from $\mu$ to $\nu$. 
\end{definition}

Following \cite{kantorovich2006translocation}, we can intuitively visualize a transportation plan $T_{\mu,\nu}$ as a way to move mass distributed over the vertices of $G$ according to $\mu$ along the edges of $G$ to an arrangement according to $\nu$. We now consider the \textit{cost} of a given transportation plan $T_{\mu,\nu}$: if moving 1 unit of mass across 1 edge has a cost of 1, how much does it cost to move the mass distribution of $\mu$ to that of $\nu$ according to $T_{\mu,\nu}$?

\begin{definition}[c.f. \cite{COTFNT}]\label{def: carrie the cost function fairy}
Define the \textbf{cost function} $C: \mathcal{T}_{\mu,\nu}\to\mathbb{R}$ to take any transportation plan $T$ to its cost $$C(T) = \sum_{(w_1,w_2) \in V(G)\times V(G)} \dist(w_1,w_2) \cdot T(w_1,w_2).$$
\end{definition}

\begin{definition} [c.f. \cite{COTFNT}] \label{def: warry the wasserstein distance fairy} 
Define the \textbf{Wasserstein distance} $W_{\geq 0}: D_{\geq 0} \to \mathbb{R}_{\geq 0}$ by $W_{\geq 0}(\mu,\nu) := \displaystyle \min_{T \in \mathcal{T}_{u,v}} C(T)$
\end{definition}

We can thus interpret the Wasserstein distance as the minimum cost of transporting mass from its arrangement in distribution $\mu$ to an arrangement in distribution $\nu$.

\begin{remark} \label{rem: addy the addition fairy}
Note that for any distribution $\psi$ on $V(G)$, if $\mu$, $\nu$, $\mu+\psi$, and $\nu+\psi$ are all nonnegative, then $W_{\geq0}(\mu,\nu) = W_{\geq0}(\mu+\psi,\nu+\psi)$ (for a proof, see for example \cite{jiradilok2021transportation}, which notes that the Wasserstein distance between $\mu$ and $\nu$ can be defined in terms of $\mu-\nu$).
\end{remark}

Let $G$ be a  graph with two distributions $\mu, \nu$ on $V(G)$ such that $$\sum_{w \in V(G)} \mu(w) = \sum_{w \in V(G)} \nu(w).$$ Let $\psi$ be a distribution such that $\mu+\psi$ and $\nu+\psi$ are both nonnegative. We extend the domain of the Wasserstein distance to include distributions $\mu,\nu$ with negative entries by defining $W(\mu,\nu): D\to \mathbb{R}_{\geq0}$ to be $W_{\geq0}(\mu+\psi,\nu+\psi)$. By Remark~\ref{rem: addy the addition fairy}, $W(\mu,\nu)$ is well-defined.

Even if $\mu$ and $\nu$ have negative entries, we can interpret $W(\mu,\nu)$ as the cost of some optimal ``transportation plan" that moves mass from distribution $\mu$ to distribution $\nu$. 

Thus, in the rest of the paper, ``transportation plans" between distributions $\mu$ and $\nu$ allow for negative entries in $\mu$ and $\nu$. In this case, a transportation plan rigorously refers to a transportation plan from $\mu + \psi$ to $\nu+\psi$ for some $\psi$ large enough that $\mu+\psi$ and $\nu+\psi$ are both nonnegative. In particular, the movement of mass between $\mu$ and $\nu$ from a vertex $w_1$ to a different vertex $w_2$ actually refers to that same movement of mass from $w_1$ to $w_2$ between the distributions $\mu+\psi$ and $\nu+\psi$. 

We now discuss a different way of calculating the Wasserstein distance.

\begin{definition}[c.f. \cite{COTFNT}]
Given a graph $G$, a \textbf{1-Lipschitz function} $\ell: V(G) \to \mathbb{R}$ is a function on the vertices of G where for any $w_1, w_2 \in V(G)$, we have that $|\ell(w_1) - \ell(w_2)| \leq \dist(w_1,w_2)$. Let $L(G)$ be the set of all 1-Lipschitz functions on $G$.
\end{definition}

\begin{theorem}[Kantorovich Duality, c.f. \cite{COTFNT}]
Let $G$ be a  graph with two distributions $\mu, \nu$ on $V(G)$ such that $\sum_{w \in V(G)} \mu(w) = \sum_{w \in V(G)} \nu(w)$. Then $$W(\mu,\nu) = \max_{\ell \in L} \sum_{w \in G} \ell(w)(\mu(w) - \nu(w)).$$
\end{theorem}

We now seek a way to refer to a pair of random walks on a graph, as these pairs of random walks are the objects we study. The information needed to define such a pair consists of the graph $G$, the starting vertices $u$ and $v$ of the two random walks, and the respective lazinesses $\alpha$ and $\beta$ of the random walks. We thus define a \textit{Guvab} comprised of this information.

\begin{definition}
We define a \textbf{Guvab} to be a tuple $(G,u,v,\alpha,\beta)$ where $G$ is a finite, connected, simple graph, $u,v \in V(G)$, and $\alpha, \beta \in [0,1]$ with $\alpha \leq \beta$. 
\end{definition}

\begin{definition}
Consider a graph $G$. For any starting vertex $u\in V(G)$ and laziness $\alpha\in[0,1]$, consider the random walk $R = \{R_k\}_{k=0}^{\infty}$ such that $R_0=u$, and, for $i\geq 1$, we have $R_i=R_{i-1}$ with probability $\alpha$, and $R_i=t$ with probability $\frac{1-\alpha}{\deg(R_{i-1})}$ for any $t\in N(R_{i-1})$. We say the probability distribution $\mu_k$ for $R_k$ is a \textbf{k-step probability measure}.
\end{definition}

Consider some Guvab $\mathcal{G} = (G,u,v,\alpha,\beta)$. We let $X(\mathcal{G}) = \{X_k\}_{k=0}^{\infty}$ be the Markov chain corresponding to a random walk with laziness $\alpha$ starting from vertex $u$ and we let $Y(\mathcal{G}) = \{Y_k\}_{k=0}^{\infty}$ be the Markov chain corresponding to a random walk with laziness $\beta$ starting from vertex $v$. When it is clear which Guvab $\mathcal{G}$ we are referring to, we write $X,Y$ instead of $X(\mathcal{G}), Y(\mathcal{G})$, respectively.

Consider some Guvab $\mathcal{G} = (G,u,v,\alpha,\beta)$. For all $k \geq 0$ we let $\mu_k(\mathcal{G}),\nu_k(\mathcal{G})$ be the k-step probability measures of $X(\mathcal{G}), Y(\mathcal{G})$ respectively. We let $\xi_k(\mathcal{G}) = \mu_k(\mathcal{G}) - \nu_k(\mathcal{G})$ and $W_k(\mathcal{G}) = W(\mu_k(\mathcal{G}), \nu_k(\mathcal{G}))$. When it is clear which Guvab $\mathcal{G}$ we are referring to, we write $\mu_k,\nu_k,\xi_k,W_k$ instead of $\mu_k(\mathcal{G}),\nu_k(\mathcal{G}),\xi_k(\mathcal{G}),W_k(\mathcal{G})$, respectively.

Given a Guvab $\mathcal{G}$, we define $P_{\alpha}$ and $P_{\beta}$ to be the transition probability matrices of $X$ and $Y$, respectively. In particular, for all $k$, we have that $\mu_k = \mu_0 P_{\alpha}^k$ and $\nu_k = \nu_0 P_{\beta}^k$, where the distributions are row vectors. We also define $P$ to be the transition probability matrix of a random walk with zero laziness on $G$ (note that $P$ does not depend on the starting vertex of the random walk). We note that $P_\alpha$ and $P_\beta$ only depend on $\alpha$ and $\beta$, not $u$ and $v$. In particular, $P_\alpha = \alpha I + (1-\alpha)P$ and $P_\beta = \beta I +(1-\beta)P$.

\begin{lemma} \label{lem: eileen the eigval sum fairy}
Let $\{\lambda_1,\ldots, \lambda_n\}$ be the union of the set of eigenvalues of $P_{\alpha}$ and the set of eigenvalues of $P_{\beta}$. For all vertices $w$, there exist some constants $c^w_i$ such that for all $k \geq 1$, we have $\xi_k(w) = \sum_{i = 1}^n c^w_i \lambda_i^k$.
\end{lemma}

\begin{proof}
This follows from the fact that $P_{\alpha}$ and $P_{\beta}$ are diagonalizable (since random walks are reversible (\cite{levin2017markov}) and thus have diagonalizable matrices (\cite{levin2017markov}, Chapter 12)). Say $P_{\alpha}$ has eigenvalues $\lambda_1, \ldots, \lambda_m$ and $P_{\beta}$ has eigenvalues $\lambda_{m+1},\ldots, \lambda_{m'}$. Since $P_{\alpha}$ is diagonalizable, we can write it as $ADA^{-1}$ for invertible matrix $A$ and diagonal matrix $D$ with diagonal entries $\lambda_1, \ldots, \lambda_m$. 

Then $\mu_k = \mu_0 P_{\alpha}^k = \mu_0 A D^k A^{-1}$, so for all $w$ there exist constants $x^w_1, \ldots, x^w_m$ such that for all $k \geq 1$ we have $\mu_k(w) = \sum_{i = 1}^m x^w_i \lambda_i^k$. By similar reasoning, for all $w$ there exist constants $y^w_{m+1}, \ldots, y^w_{m'}$ such that for all $k \geq 1$ we have $\nu_k(w) = \sum_{i = m+1}^{m'} y^w_i \lambda_i^k$. Therefore, for all $w$, there exist some constants $c^w_i$ such that for all $k \geq 1$, we have that $\xi_k(w) = \sum_{i = 1}^{m'} c^w_i \lambda_i^k$. If for any $i$ and $j$ we have $\lambda_i = \lambda_j$, we can collect these like terms and thus create a list of distinct eigenvalues $\lambda_1,\ldots, \lambda_n$ and constants $c^w_1,\ldots c^w_n$ such that for all $k \geq 1$, we have $\xi_k(w) = \sum_{i = 1}^n c^w_i \lambda_i^k$. In particular, $\lambda_1,\ldots, \lambda_n$ will be exactly the elements of the union of the set of eigenvalues of $P_\alpha$ and the set of eigenvalues of $P_\beta$. 
\end{proof}

In the next section, we discuss when and how the Wasserstein distance converges, which is related to the convergence of probability distributions of random walks. Since random walks can be viewed as Markov chains, we reference some classical Markov chain theory, using the same definitions as in \cite{levin2017markov}. We also use the following well-known Markov chain theorem.

\begin{theorem} [c.f. \cite{levin2017markov}] \label{thm: aperiodic irreducible}
Suppose that a Markov chain $X$ is aperiodic and irreducible with probability distributions $(\mu_0,\mu_1,\ldots)$ and stationary distribution $\pi$. Then $\lim_{k\to\infty}\mu_k = \pi$.
\end{theorem}

Finally, in our discussion of convergence, we encounter cases where the Wasserstein distance is eventually constant. To quantify this precisely, we provide the following definition.

\begin{definition}
We call an infinite sequence $\{S_i\}_{i=0}^{\infty}$ for $S_i \in \mathbb{R}$ \textbf{eventually constant} if there exists $N \geq 0$ such that for all $k \geq N$, we have that $S_k = S_N$.
\end{definition}

\section{Classifying End Behavior of $W_k$}

In this section, we seek to enumerate the possible end behaviors of the Wasserstein distance for a Guvab. In particular, we prove results about when the Wasserstein distance converges and what it converges to for different Guvabs. The classification of Guvabs by end behavior paves the way for our later discussion of the rate of convergence of the Wasserstein distance.

We begin with a technical lemma showing that the limit of the Wasserstein distance is the Wasserstein distance of the limit, as we expect.

\begin{lemma}\label{lem: converges to stationary distance}
Let $f: \mathbb{Z}_{\geq 0} \to \mathbb{Z}_{\geq 0}$ be a strictly increasing function. If $\displaystyle\lim_{k\to\infty}\mu_{f(k)}=\mu$ and $\displaystyle\lim_{k\to\infty}\nu_{f(k)}=\nu$ (and, in particular, both limits exist), then $$\lim_{k\to\infty}W(\mu_{f(k)},\nu_{f(k)})=W(\mu,\nu).$$
\end{lemma}

\begin{proof}
Note that, by the triangle inequality, $$W(\mu_{f(k)},\nu_{f(k)})\leq W(\mu_{f(k)},\mu)+W(\mu,\nu)+W(\nu,\nu_{f(k)})$$ and $$W(\mu,\mu_{f(k)})+W(\mu_{f(k)},\nu_{f(k)})+W(\nu_{f(k)},\nu)\geq W(\mu,\nu).$$ This implies that
\begin{align*}
W(\mu,\nu)-W(\mu,\mu_{f(k)})-W(\nu_{f(k)},\nu)
&\leq W(\mu_{f(k)},\nu_{f(k)})\\ 
&\leq W(\mu_{f(k)},\mu)+W(\mu,\nu)+W(\nu,\nu_{f(k)}). 
\end{align*}
However, $$\displaystyle\lim_{k\to\infty}W(\mu,\mu_{f(k)})=\lim_{k\to\infty}W(\mu-\mu_{f(k)},0)=0$$ (and similarly for $W(\nu,\nu_{f(k)})$). The above inequality implies that $$\lim_{k\to\infty}W(\mu_{f(k)},\nu_{f(k)})=W(\mu,\nu),$$ as desired.
\end{proof}

Due to classical Markov chain theory, we expect that in most cases, the probability distributions of both random walks converge to the same stationary distribution, and thus $\lim_{k\to\infty} W_k = 0$. The following definition and lemma quantify the stationary distribution that most random walks converge to. The subsequent theorem specifies what the ``most cases'' in which the distance goes to zero are.

\begin{definition}
For any graph $G$, we define the distribution $\pi$ to be such that for any $i \in G$, we have $\pi_i=\displaystyle \frac{\deg(i)}{\sum_{j\in G}\deg(j)}.$ 
\end{definition}

\begin{lemma} \label{lem: pippa the pi fairy}
When $0 < \alpha < 1$, the k-step probability measure $\mu_k$ converges to the stationary distribution $\pi$.
\end{lemma}

\begin{proof}
Recall that $X$ is the Markov chain of the random walk. We have that $X$ is aperiodic (we can return from a vertex to itself in one step) and irreducible ($G$ is connected). We have that for any vertex $w\in G$, $$\pi_w = \displaystyle \sum_{i \sim w}\pi_i\frac{1}{\deg(i)} = \displaystyle \alpha\pi_w + \sum_{i \sim w}\pi_i\frac{1-\alpha}{\deg(i)}.$$ Thus, $\pi$ is a stationary distribution of $X$. Hence, by Theorem \ref{thm: aperiodic irreducible}, $\pi$ is a limiting distribution for $X$ and thus $\lim_{k\to\infty}\mu_k = \pi$.
\end{proof}

\begin{theorem}\label{thm: 0-convergence}
The value $W(\mu_k,\nu_k)$ converges to $0$ as $k \to \infty$ if and only if one of the following conditions is true:
\begin{itemize}
    \item $0 < \alpha\leq \beta < 1$,
    \item $\alpha = \beta = 1$ and $u = v$,
    \item $G$ is not bipartite and $0=\alpha\leq \beta < 1$,
    \item $\alpha = \beta = 0$ and there exists a path from $u$ to $v$ with an even number of steps.
\end{itemize}
\end{theorem}

\begin{proof}
Note that for $0 < \alpha\leq \beta < 1$ , we have by Lemma \ref{lem: pippa the pi fairy} that $$\lim_{k\to\infty}\mu_k = \lim_{k\to\infty}\nu_k = \pi.$$ Thus, $\lim_{k\to\infty} W(\mu_k,\nu_k)=0$ in this case.

We now consider the cases where $\alpha=0$ or $\beta = 1$. If $\beta=1$, then $Y$ stays at $v$ forever. Thus, in order to have $\lim_{k\to\infty} W(\mu_k,\nu_k)=0$, we need $\alpha=1$ and $u=v$. This is sufficient to imply $\lim_{k\to\infty} W(\mu_k,\nu_k)=0$. 

It remains to look at the case where $\alpha=0$ and $\beta< 1$, which we break into subcases based on whether $G$ is bipartite. 

We first tackle the subcase where $G$ is not bipartite, i.e., $G$ contains an odd cycle. Since $\alpha,\beta< 1$, both $X,Y$ are aperiodic (there is a path from any vertex to itself in both an odd number of steps and an even number of steps via the odd cycle) and irreducible ($G$ is connected). Thus, $\lim_{k\to\infty}\mu_k=\lim_{k\to\infty}\nu_k=\pi$ and $\lim_{k\to\infty} W(\mu_k,\nu_k)=0$ as before.

Finally, we address the subcase where $G$ is bipartite with sides $S_1,S_2$. Here, $X$ is periodic (with period 2), so $Y$ must be periodic as well to have $\lim_{k\to\infty} W(\mu_k,\nu_k)=0$. Thus, $\beta=0$. If $u,v$ are on different sides of $G$, then $X$ and $Y$ will never be on the same side, so we cannot have $\lim_{k\to\infty} W(\mu_k,\nu_k)=0$. Otherwise, without loss of generality let $u,v \in S_1$. Consider the Markov chains $X' =\{X_{2k}\}_{k=0}^{\infty}$ and $Y' =\{Y_{2k}\}_{k=0}^{\infty}$ with vertex set $S_1$. Since $X'$ and $Y'$ are aperiodic (we can get from a vertex to itself in one step of $X'$ or $Y'$ by moving back and forth along the same edge) and irreducible ($G$ is connected), and they both have the same transition matrix, the Markov chains converge to the same stationary distribution. Similar reasoning applies for $\{X_{2k+1}\}$ and $\{Y_{2k+1}\}$. This finishes the proof for this case, hence completing the proof of Theorem~\ref{thm: 0-convergence}.
\end{proof}

In the next part of this section, we specify what the stationary distributions look like for any possible Guvab, particularly considering Guvabs with more than one stationary distribution $\pi$. We show that all Guvabs either have one set of end behaviors they converge to or switch back and forth between two sets of end behaviors.

Suppose $\alpha = 0$. Let $G$ be bipartite with sides $S_1,S_2$, and without loss of generality let $u\in S_1$. Let $X_1 = \{X_{2k}\}_0^{\infty}$ and $X_2 = \{X_{2k+1}\}_0^{\infty}$. 
Let $X_i'$ denote $X_i$ restricted to $S_i$ for $i\in\{1,2\}.$ For $i\in\{1,2\}$, let $\tau_i'$ be a distribution on $S_i$ such that $$(\tau_i')_w=\frac{2\deg(w)}{\sum_{j \in G}\deg(j)}$$ for $w \in S_i$. Further, for $i \in \{1,2\}$,  let $\tau_i$ be a distribution on $G$ that is $\tau_i'$ on $S_i$ and has value 0 elsewhere.

\begin{lemma} \label{lem: bipartite limiting distribution}
For $i\in\{1,2\}$, the distribution $\tau_i'$ is the limiting distribution of $X_i'$. 
\end{lemma}

\begin{proof}
First, we claim that $\tau_1P^2=\tau_1$ and $\tau_2P^2=\tau_2$, where $P$ is the transition matrix of $X$. Note that for $w \in S_2$, we have 
$$\displaystyle(\tau_1P)_w= \sum_{i\sim w} \frac{1}{\deg(i)}(\tau_1)_i = \sum_{i\sim w} \left(\frac{1}{\deg(i)}\right)\left(\frac{\deg(i)}{\sum_{j\in G}\deg(j)}\right)=\frac{\deg(w)}{\sum_{j\in G}\deg(j)}=(\tau_2)_w.$$ 
This is because for all $i \sim w$, we have $i \in S_1$, which implies $(\tau_1)_i = \frac{\deg(i)}{\sum_{j\in G}\deg(j)}$. For $w\in S_1$, we have $\displaystyle(\tau_1P)_w=\sum_{i\sim w} \frac{1}{\deg(i)}(\tau_1)_i=0=(\tau_2)_w$. This is because for all $i \sim w$, we have $i \in S_2$, which implies $(\tau_1)_i = 0$. Hence, $\tau_1P=\tau_2$ and, by similar reasoning, $\tau_2P=\tau_1$. Thus, $\tau_1P^2=\tau_1$ and $\tau_2P^2=\tau_2$ as desired.

Also, we note that $\displaystyle \sum_{w\in S_i} (\tau_i')_w = \sum_{w\in S_i}\frac{2\deg(w)}{\sum_{j \in G}\deg(j)} = \frac{2|E(G)|}{\sum_{j \in G}\deg(j)} = 1.$

We now see that $\tau_i'$ is a stationary distribution of $X_i'$ for $i\in\{1,2\}$. Since $X_1'$ and $X_2'$ are irreducible and aperiodic (as shown in the proof of Theorem~\ref{thm: 0-convergence}), we have that $\tau_i'$ is a limiting distribution of $X_i'$ for $i\in\{1,2\}$.
\end{proof}

\begin{corollary}\label{cor: bipartite convergence}
If $u\in S_1$, then as $k \to \infty$, we have $\mu_{2k}$ converges to $\tau_1$ and $\mu_{2k+1}$ converges to $\tau_2$. Analogously, if $u\in S_2$ then as $k \to \infty$, we have $\mu_{2k}$ converges to $\tau_2$ and $\mu_{2k+1}$ converges to $\tau_1$.
\end{corollary}

\begin{proof}
Suppose $u\in S_1$; the proof will proceed analogously if $u\in S_2$. Then, $\mu_{2k}$ will always be 0 on $S_2$ and, by Lemma \ref{lem: bipartite limiting distribution}, it will converge to $\tau_1'$ on $S_1$ because $\mu_{2k}$ is the probability distribution of $X_1'$ on $S_1$. Similarly, $\mu_{2k+1}$ will always be 0 on $S_1$ and it will converge to $\tau_2'$ on $S_2$. Thus, $\mu_{2k}$ converges to $\tau_1$ and $\mu_{2k+1}$ converges to $\tau_2$.
\end{proof}

\begin{corollary}\label{corolawrence: xi0,xi1 are well defined}
For any Guvab, $\displaystyle \lim_{k \to \infty} \xi_{2k}$ and $\displaystyle \lim_{k \to \infty} \xi_{2k+1}$ are well-defined.
\end{corollary}

\begin{proof}
We show that for any $\mu$, we have that $\lim_{k\to\infty} \mu_{2k}$ and $\lim_{k\to\infty} \mu_{2k+1}$ are well-defined; this implies the statement of the corollary. When $G$ is bipartite and $\alpha = 0$, we know that $\lim_{k\to\infty} \mu_{2k} = \tau_1$ and $\lim_{k\to\infty} \mu_{2k+1} = \tau_2$ (assuming, without loss of generality, that $u\in S_1$). When $\alpha = 0$ and $G$ is not bipartite or when $0 < \alpha < 1$, we have 
$$\lim_{k\to\infty} \mu_{2k} = \lim_{k\to\infty} \mu_{2k+1} = \pi.$$ Finally, when $\alpha = 1$, we know $\lim_{k\to\infty} \mu_{2k} = \lim_{k\to\infty} \mu_{2k+1} = \mathbbm{1}_u$. This covers all possible cases for $\alpha$ and $G$, so we are done.
\end{proof}

For any Guvab, we refer to $\displaystyle \lim_{k \to \infty} \xi_{2k}$ as $\xi^0$ and $\displaystyle \lim_{k \to \infty} \xi_{2k+1}$ as $\xi^1$.

The following corollary is quite important for the rest of this section and the remainder of this paper. Its relevance to this section is that $\lim_{k\to\infty} W_k$ will be well-defined unless $\lim_{k\to\infty} W_{2k} \neq \lim_{k\to\infty} W_{2k+1}$. The corollary is pertinent to the rest of the paper because it indicates that the rates of convergence of $\{W_{2k}\}$ and $\{W_{2k+1}\}$ are always well-defined. Thus, for any possible Guvab, we can study and state results about the rates of convergence of $\{W_{2k}\}$ and $\{W_{2k+1}\}$.

\begin{corollary}
We have that $\lim_{k\to\infty} W_{2k}$ and $\lim_{k\to\infty} W_{2k+1}$ are always well-defined.
\end{corollary}
\begin{proof}
We know that $\lim_{k\to\infty} W_{2k} = W(\xi^0,\Tilde{\textbf{0}})$ and $\lim_{k\to\infty} W_{2k+1} = W(\xi^1,\Tilde{\textbf{0}})$.
\end{proof}

We soon discuss many cases where $\lim_{k\to\infty} W_k$ exists, so we designate a way to refer to this limit. For any Guvab $\mathcal{G}$ where $\lim_{k\to\infty} W_k$ exists, we denote by $W$ the limit $\lim_{k\to\infty} W_k$.

We can now state and prove our main theorems about whether the Wasserstein distance converges and the values it converges to. For any possible Guvab, Theorem \ref{thm: convergence condition on mildly vexing scenarios} allows us to determine whether the Wasserstein distance converges. Furthermore, Theorem \ref{thm: convergence values} allows us to, in most cases, quickly and easily determine what value the Wasserstein distance will converge to. Finally, these theorems provide a framework for us to classify the Guvabs into four categories so we can use casework to understand the rate of convergence.

\begin{theorem}\label{thm: convergence values}
Unless $G$ is bipartite, $\alpha=0$, and $\beta=1$, we have that $\displaystyle W=\lim_{k\to\infty}W(\mu_k,\nu_k)$ is always well defined, and furthermore
\begin{itemize}
    \item $W=0$ under the conditions specified in Theorem~\ref{thm: 0-convergence},
    \item $W=1$ if $\alpha=\beta=0$ and $W\neq 0$,
    \item $W=\frac{1}{2}$ if $0=\alpha< \beta < 1$ and $G$ is bipartite.
\end{itemize}
\end{theorem}

\begin{proof}
The first condition is clear by Theorem~\ref{thm: 0-convergence}. Next, we look at the case where $\alpha=\beta=0$ and $W\neq 0$. By Theorem~\ref{thm: 0-convergence}, this corresponds to the case where $G$ is bipartite and $u,v$ are on opposite sides of $G$. Without loss of generality, let $u\in S_1$ and $v\in S_2$. Then, as $k \to \infty$, we have $\mu_{2k}$ converges to $\tau_1$ and $\nu_{2k}$ converges to $\tau_2$ by Corollary~\ref{cor: bipartite convergence}. Analogously, $\mu_{2k+1}$ converges to $\tau_2$ and $\nu_{2k+1}$ converges to $\tau_1$. Thus, $\displaystyle \lim_{k \to \infty} W(\mu_k, \nu_k) = W(\tau_1, \tau_2)$. We have that $W(\tau_1, \tau_2) \geq 1$ because to get from $\tau_1$ to $\tau_2$, we must move all the mass from $S_1$ across at least one edge to $S_2$. Also, $W(\tau_1, \tau_2) \leq 1$ because we can achieve a distance of 1 by, for any given edge $ab$ with $a \in S_1$ and $b \in S_2$, moving a mass of $\displaystyle \frac{2}{\sum_{j \in G}\deg(j)}$ from $a$ to $b$. 

We now consider the case when $0 = \alpha < \beta < 1$ and $G$ is bipartite. Without loss of generality, let $u\in S_1$. Since $\alpha=0$, we have that $\lim_{k\to\infty}\mu_{2k} = \tau_1$ and $\lim_{k\to\infty}\mu_{2k+1} = \tau_2$. Since $\beta > 0$, we have that $\lim_{k\to\infty}\nu_k = \pi$. Thus, we have that $\displaystyle \lim_{k \to \infty} W(\mu_{2k}, \nu_{2k}) = W(\tau_1,\pi)$ and $\displaystyle \lim_{k \to \infty} W(\mu_{2k+1}, \nu_{2k+1}) = W(\tau_2,\pi)$. If we show that $W(\tau_1,\pi) = W(\tau_2,\pi) = \frac{1}{2}$, we will have shown the third condition. We know that $\pi$ will have half its mass on $S_1$ and half its mass on $S_2$ because $$\sum_{v \in S_1} \pi_v = \sum_{v \in S_1} \frac{\deg(v)}{\sum_{j \in G}\deg(j)} = \frac{|E(G)|}{\sum_{j \in G}\deg(j)} = \frac{1}{2}.$$ Thus, half the mass must move from $S_2$ to $S_1$, so $W(\pi, \tau_1) \geq \frac{1}{2}$. We can also achieve a distance of exactly $\frac{1}{2}$ from $\pi$ to $\tau_1$ by, for any given edge $ab$ with $a \in S_1$ and $b \in S_2$, moving $\displaystyle \frac{1}{\sum_{j \in G}\deg(j)}$ mass from $b$ to $a$. Thus, $W(\pi,\tau_1) = \frac{1}{2}$ and by an analogous argument, $W(\pi,\tau_2) = \frac{1}{2}$.

We have now considered all cases where $\alpha,\beta < 1$ and where $\alpha = \beta = 1$. The only case left is where $0 < \alpha < 1$ and $\beta = 1$. Here, $\lim_{k \to \infty} \mu_k = \pi$ and $\nu_k = \mathbbm{1}_v$ where $\mathbbm{1}_v$ is the distribution with 1 at $v$ and 0 elsewhere. Thus, $\lim_{k \to \infty} W(\mu_{k}, \nu_{k}) = W(\pi, \mathbbm{1}_v)$, which is a constant.
\end{proof}

\begin{theorem} \label{thm: convergence condition on mildly vexing scenarios}
The distance $W(\mu_k,\nu_k)$ does not converge as $k \to \infty$ if and only if $G$ is bipartite, $\alpha = 0$ and $\beta = 1$, and $$\sum_{w \in V(G)} (-1)^{\emph{\dist}(v,w)}\emph{\dist}(v,w)\deg(w) \neq 0.$$
\end{theorem}

\begin{proof}
By Theorem~\ref{thm: convergence values}, we know that the only case where it is possible for $W(\mu_k,\nu_k)$ not to converge is when $G$ is bipartite, $\alpha = 0$, and $\beta = 1$. In this case, $\nu_k = \mathbbm{1}_v$. Additionally, assuming without loss of generality that $u \in S_1$, we have that 
$$\lim_{k\to\infty}\mu_{2k} = \tau_1 \text{ and } \lim_{k\to\infty}\mu_{2k+1} = \tau_2.$$ Thus, $W(\mu_k,\nu_k)$ converges as $k \to \infty$ if and only if $W(\mathbbm{1}_v,\tau_1) = W(\mathbbm{1}_v,\tau_2)$. 

To calculate $W(\mathbbm{1}_v,\tau_1)$, we note that we must move all the mass of $\tau_1$ to vertex $v$. To move all the mass at some vertex $w$ to $v$, we necessarily move a mass of $(\tau_1)_w$ over a distance of $\dist(w,v)$. Thus the total transportation cost, and thus the total Wasserstein distance $W(\mathbbm{1}_v,\tau_1)$, is given by 
\begin{align*}
    \sum_{w \in G} \dist(v,w)(\tau_1)_w &= \sum_{w \in S_1} \dist(v,w)\frac{2\deg(w)}{\sum_{j \in G}\deg(j)} + \sum_{w \in S_2} \dist(v,w)\cdot 0\\ &= \frac{2}{\sum_{j \in G}\deg(j)}\sum_{w \in S_1} \dist(v,w)\deg(w).
\end{align*}
By the same reasoning, we have that $$ W(\mathbbm{1}_v,\tau_2) = \frac{2}{\sum_{j \in G}\deg(j)}\sum_{w \in S_2} \dist(v,w)\deg(w).$$ Given that $G$ is bipartite, $\alpha = 0$, and $\beta = 1$, we know that the Wasserstein distance converges if and only if $W(\mathbbm{1}_v,\tau_1) - W(\mathbbm{1}_v,\tau_2) = 0$, which is true if and only if $$\sum_{w \in V(G)} (-1)^{\dist(v,w)}\dist(v,w)\deg(w) = 0$$
since the parity of $\dist(v,w)$ depends only on the side of $G$ that $w$ is on. Thus, the theorem statement follows.
\end{proof}

We now present a table summarizing much of the information about convergence discussed in this section.

\begin{center}
\begin{table}[H]
\small\addtolength{\tabcolsep}{-5pt}
\begin{tabular}{c|c|c|c|c|c}
 Conditions on $\mathcal{G}$ & $W = 0$ & $W = \frac{1}{2}$ & $W = 1$ & $W = C\neq 0,\frac{1}{2},1$ & $W_k$ does not converge \\ \hline
 $G$ bipartite, $\beta=1$ & $\checkmark$ & $\checkmark$ & $\checkmark$ & $\checkmark$ & $\checkmark$ \\
 $G$ bipartite, $\beta<1$ & $\checkmark$ & $\checkmark$ & $\checkmark$ & $\times$ & $\times$ \\
 $G$ non-bipartite, $\beta=1$ & $\checkmark$ & $\times$ & $\checkmark$ & $\checkmark$ & $\times$ \\
 $G$ non-bipartite, $\beta<1$ & $\checkmark$ & $\times$ & $\times$ & $\times$ & $\times$ \\
\end{tabular}
\caption{Is it possible for the Wasserstein distance to converge to particular limits in different cases of conditions on $\mathcal{G}$?}
\label{tab: The Guvable}
\end{table}
\end{center}
\vspace{-0.4in} 

\begin{remark}
We know that the case of $G$ bipartite, $\beta = 1$, and $W = \frac{1}{2}$ is possible by considering a star with $v$ at the center and $0 < \alpha < 1$. We know that the case of $G$ non-bipartite, $\beta = 1$, and $W = \frac{1}{2}$ is impossible because in order for it to be possible, $\lim_{k\to\infty}\mu_k$ would need half of its mass to be at $v$. Since mass of $\lim_{k\to\infty}\mu_k$ is proportional to degree, every edge would have to be incident to $v$, making the graph bipartite.
\end{remark}

The following corollary provides a categorization of the Guvabs into four types. In the next four sections of this paper, we examine each of these categories in turn. 

\begin{corollary} \label{cor: four guvab nations}
Each Guvab satisfies exactly one of the following four conditions: 
\begin{itemize}
    \item $W=1$ and $\beta < 1$,
    \item $W=\frac{1}{2}$ and $\beta < 1$,
    \item $W = 0$ and $\beta < 1$,
    \item $\beta = 1$.
\end{itemize}
\end{corollary}

\begin{proof}
If $\beta < 1$, we have $W=0$ under the conditions in Theorem \ref{thm: 0-convergence} and $W = 1$ or $W =\frac{1}{2}$ otherwise, since the conditions in Theorem \ref{thm: convergence values} cover all possible cases where $\beta < 1$ and $W \neq 0$.
\end{proof}

If we understand the convergence of the Wasserstein distance in all four of these cases, then we understand the convergence for all Guvabs. The subsequent four sections each discuss the convergence of the Wasserstein distance in one of these cases. Our two main convergence theorems, presented in Section 8, put together the general results obtained by examining these four cases individually.

\section{Convergence when $W = 1$}
In this section we consider Guvabs with $W = 1$ and $\beta < 1$. Recall that these are exactly the Guvabs for which $G$ is bipartite, $u$ and $v$ are on different sides of the bipartite graph, and $\alpha = \beta = 0$. We show that all such Guvabs have a Wasserstein distance that is eventually constant. We also begin to understand how long it takes for the Wasserstein distance to reach constancy.

We first recall that the Wasserstein distance between two distributions $\mu$ and $\nu$ with potentially negative entries is the cost of an optimal transportation plan for moving mass\footnote{As discussed in section 2, the mass of a distribution $\mu$ at a vertex $w$ is $\mu(w)$, the value of the distribution at that vertex.} from $\mu$ to $\nu$. Thus, to prove the eventual constancy of the Wasserstein distance, we construct an algorithm that produces a transportation plan between any two distributions. Then, we show that when certain inequalities are satisfied, this transportation plan has a cost of exactly 1 and is optimal. Finally, we prove that when $\xi_k$ is eventually sufficiently close to either of the stationary distributions $\xi^0$ or $\xi^1$, these inequalities are satisfied.

We start by constructing the algorithm.
Pick a spanning tree $T$ of $G$ and let $L$ be the set of leaves of $T$. Define a function $r: V(G)\to \mathbb{Z}$ 
such that $r(w) = \min_{\ell \in L}\dist(w,\ell)$. 

For any finite set $S$, let $\Perm(S)$ denote the set of all permutations of $S$. We say that an \textbf{$r$-monotone ordering} $\mathcal{O} = (w_1, \ldots, w_n)\in \Perm(V(G))$ is a permutation of $V(G)$ such that $r(w_1), \ldots, r(w_n)$ is a non-decreasing sequence.

\begin{definition} \label{Allie the Albatross Algorithm Fairy}
Given a graph $G$, a spanning tree $T$ of $G$, an $r$-monotone ordering $\mathcal{O}$ and zero-sum distribution $\xi$, we define the \textbf{tree-based transport algorithm}, which transports mass from $\xi$ to $\Tilde{\textbf{0}}$, to be an $(n-1)$-step algorithm in which at the $i$th step, 
\begin{itemize}
    \item if the current mass at $w_i$ is nonnegative, we distribute it evenly among all $v\sim w_i$ with indices greater than $i$,
    \item if the current mass at $w_i$ is negative, we take an equal amount of mass to vertex $w_i$ from all $v\in N(w_i)$ with indices greater than $i$, so that the mass at $w_i$ is now $0$.
\end{itemize}
\end{definition}

In Lemma \ref{aggie the agtoe fairy}, we see that this algorithm produces a valid transportation plan from $\xi$ to $\Tilde{\textbf{0}}$. We refer to this \textbf{tree-based transportation plan} as $A(G,T,\mathcal{O},\xi)$. Given $G, T$ and $\mathcal{O}$, we let $A_i(\xi)$ denote the distribution of mass on the vertices of $G$ after $i$ steps of the algorithm.

\begin{lemma}\label{aggie the agtoe fairy}
The tree-based transport algorithm on $G, T, \mathcal{O}, \xi$ always produces a valid transportation plan from $\xi$ to $\Tilde{\textbf{0}}$.
\end{lemma}

\begin{proof}
After the $i$th step of the tree-based transport algorithm, the mass at each of the vertices $w_1, \ldots, w_i$ is $0$, since the mass at $w_j$ becomes zero at the $j$th step, and thereafter no mass is moved to or from $w_j$. Thus, after the $(n-2)$th step, the only vertices of $G$ with nonzero mass will be $w_{n-1}$ and $w_n$. Since the total mass sums to $0$ and $w_{n-1}$ is adjacent to $w_n$, the $(n-1)$th step of the algorithm simply moves the positive mass to the negative mass so that all vertices have mass $0$.
\end{proof}

We now prove a useful property of this algorithm.
\begin{lemma} \label{lem: linear}
Given a graph $G$, tree $T$ and $r$-monotone ordering $\mathcal{O}$, for all $i$, we have that $A_i$ is a linear function on the space of zero-sum distributions.
\end{lemma}

\begin{proof}
It suffices to show that for any two zero-sum distributions $\xi$ and $\xi'$, we have $A_i(\xi + \xi') = A_i(\xi) + A_i(\xi')$.
We prove this by induction on $i$.

Base case: When $i=0$, we have that $A_0(\xi + \xi') = \xi + \xi' = A_0(\xi) + A_0(\xi')$.

Inductive step: For the inductive hypothesis, we assume that $A_i(\xi + \xi') = A_i(\xi) + A_i(\xi')$. We want to show that $A_{i+1}(\xi + \xi') = A_{i+1}(\xi) + A_{i+1}(\xi')$. For any distribution $\xi$, if $n$ denotes the number of neighbors of $w_{i+1}$ with indices greater than $i+1$, then all of the following are true:
\begin{itemize}
    \item $A_{i+1}(\xi)(w_{i+1}) = 0$,
    \item for $w_j\in N(w_{i+1})$ with $j > i+1$, we have $A_{i+1}(\xi)(w_j) = A_i(\xi)(w_j) + \frac{1}{n} A_i(\xi)(w_{i+1})$,
    \item for all other vertices $w$, we have $A_{i+1}(\xi)(w) = A_i(\xi)(w)$.
\end{itemize}

Thus, $A_{i+1}(\xi + \xi')(w_{i+1}) = 0 = 0 +0 = A_{i+1}(\xi)(w_{i+1}) + A_{i+1}(\xi')(w_{i+1})$. For $w_j\in N(w_{i+1})$ with $j > i+1$, we have that 
\begin{align*}
A_{i+1}(\xi + \xi')(w_j) &= A_i(\xi + \xi')(w_j) + \frac{1}{n} A_i(\xi + \xi')(w_{i+1})\\
&= A_i(\xi)(w_j) + \frac{1}{n} A_i(\xi)(w_{i+1}) + A_i(\xi')(w_j) + \frac{1}{n} A_i(\xi')(w_{i+1})\\
&= A_{i+1}(\xi)(w_j) + A_{i+1}(\xi')(w_j).
\end{align*}
Finally, for all other vertices $w$, we have that 
$$A_{i+1}(\xi + \xi')(w) = A_i(\xi + \xi')(w) = A_i(\xi)(w) + A_i(\xi')(w) = A_{i+1}(\xi)(w) + A_{i+1}(\xi')(w).$$ We have shown $A_{i+1}(\xi + \xi') = A_{i+1}(\xi) + A_{i+1}(\xi')$ for all the vertices, so we have proven the inductive step and thus the lemma.
\end{proof}

In Definition \ref{def: inequalities} and the subsequent results, we define the inequalities used in conjunction with the tree-based transport algorithm and show that when these inequalities are satisfied, the Wasserstein distance between $\xi$ and $\Tilde{\textbf{0}}$ will be 1.

\begin{definition} \label{def: inequalities}
For any graph $G$, zero-sum distribution $\xi$, spanning tree $T$, and $r$-monotone ordering $\mathcal{O} = (w_1,\ldots w_n)$ on $V(G)$, define the \textbf{tree-based transport inequalities} $\mathcal{I}(G,T,\mathcal{O},\xi)$ to be the union of the following two sets of inequalities: 

\begin{itemize}
    \item $I_1$: the set of inequalities of the form $\xi(w_j)A_i(\xi)(w_j) > 0$ for all $0 \leq i \leq |V(G)|-2$ and $i < j \leq V(G)$,
    \item $I_2$: the set of inequalities of the form $\xi(t)\xi(w) < 0$ for all $t \sim w$
\end{itemize}
\end{definition}

\begin{lemma} \label{lem: dist is sum of pos mass}
If the tree-based transport inequalities $\mathcal{I}(G,T,\mathcal{O},\xi)$ are satisfied, then the cost of the transportation plan  is at most the sum of positive mass in $\xi$, i.e., $\displaystyle C(A(G,T,\mathcal{O},\xi))\leq\frac{1}{2} \sum_{w \in G} |\xi(w)|.$
\end{lemma}

\begin{proof}
We note that the inequalities in $I_1$ mean that for any vertex $w_j$, the sign of $A_i(\xi)(w_j)$ stays the same until the mass becomes 0 at the $j$th step, at which point it remains 0 for the rest of the algorithm. 

Since only positive mass moves, it suffices to show that all the positive mass of $\xi$ moves a distance of at most 1. At each step of the tree-based transport algorithm, any mass that moves must move a distance of exactly 1. Thus, it suffices to show that all mass moves at most one time in $A(G,T,\mathcal{O},\xi)$. 

To show this, we demonstrate that the mass that moves at each step of the algorithm has not moved before, since this means all mass moves at most once overall. We begin by demonstrating that every time positive mass moves, it moves from a vertex $w$ for which $\xi(w) > 0$. 

The only way for mass to move is via the $i$th step of the tree-based transport algorithm, which starts from the distribution $A_{i-1}(\xi)$. Suppose the vertices are $w_1,\ldots w_n$. If $A_{i-1}(\xi)(w_i)$ is zero, then no mass moves on the $i$th step. If $A_{i-1}(\xi)(w_i)$ is positive, then at the $i$th step, mass moves away from $w_i$. In addition, by the inequalities in $I_1$, if $A_{i-1}(\xi)(w_i)$ is positive then $\xi(w_i)$ is positive and if $A_{i-1}(\xi)(w_i)$ is negative then $\xi(w_i)$ is negative. Thus, by the inequalities in $I_2$, we have $\xi(t) > 0$ for all $t \sim w_i$. If $A_{i-1}(\xi)(w_i)$ is negative, mass moves from these $t \sim w_i$ to $w_i$, so all mass movements are from a vertex $t$ for which $\xi(t) > 0$. Thus, in all three of these cases, every time positive mass moves, it moves from a vertex $w$ for which $\xi(w) > 0$.

Thus, consider the $i$th vertex, call this $w$, and suppose that $\xi(w) > 0$. Then, by the inequalities in $I_2$, we know $\xi$ has negative mass at the neighbors of $w$, so throughout all steps of the algorithm, the mass at the neighbors of $w$ was nonpositive. This means that anytime executing a step for one of the neighbors of $w$ changed the mass at $w$, mass moved from $w$ to its neighbors. Because no mass moved from another vertex to $w$, any remaining positive mass at $w$ has not yet moved. We also know that the remaining mass at $w$ is always nonnegative by the inequalities in $I_1$. Thus, whenever we execute a step of the algorithm for one of the neighbors of $w$, the nonnegative mass that moves from $w$ has not yet moved. 

Furthermore, during the $i$th step of the algorithm, all the remaining nonnegative mass at $w$ moves away from it, and this mass has not yet moved. Mass movements due to steps of the algorithm for neighbors of $w$ and due to the $i$th step, which is for $w$, make up all possible movements of the mass initially at $w$. This argument holds for all vertices $w$ for which $\xi(w) > 0$, so all possible movements of positive mass move mass that has not been moved before. Thus, we are done. 
\end{proof}

\begin{corollary} \label{cor: when ineqs, dist is sum of pos mass}
For any graph $G$ and zero-sum distribution $\xi$, if for some $T$ and $\mathcal{O}$ the tree-based transport inequalities $\mathcal{I}(G,T,\mathcal{O},\xi)$ are satisfied, then 
$$\displaystyle W(\xi,0) = \frac{1}{2} \sum_{w \in G} |\xi(w)|.$$
\end{corollary}

\begin{proof}
By Lemma \ref{lem: dist is sum of pos mass}, we have that $\frac{1}{2} \sum_{w \in G} |\xi(w)|$, the sum of positive mass, is the upper bound. For the lower bound, we note that all positive mass must move because we only move positive mass. Thus, all positive mass must move at least a distance of 1, so $W(\xi,0)$ will be at least the sum of positive mass.
\end{proof}

\begin{corollary} \label{cor: when ineqs, dist is 1}
For a Guvab $\mathcal{G}$ where $W = 1$ and $\beta < 1$, suppose that there exists some spanning tree $T$ of $G$ and $r$-monotone ordering $\mathcal{O}$ such that $\xi_k$ satisfies the tree-based transport inequalities $\mathcal{I}(G,T,\mathcal{O},\xi_k)$. Then $W_k(\mathcal{G}) = 1$.
\end{corollary}

\begin{proof}
Recall that when $W = 1$ and $\beta < 1$, we must have that $G$ is bipartite, $u$ and $v$ are on different sides of the bipartite graph, and $\alpha = \beta = 0$. Thus, for all $k$ we have that $\mu_k$ and $\nu_k$ are nonzero on disjoint sets of vertices, since at all times $\mu_k$ is nonzero only on one side and $\nu_k$ is nonzero only on the other side. Thus $\sum_{w \in G} |\xi(w)| = 2$, so by Corollary \ref{cor: when ineqs, dist is sum of pos mass} we have that $$W_k(\mathcal{G}) = W(\xi_k,0) = \frac{1}{2} \sum_{w \in G} |\xi(w)| = 1.$$
\end{proof}

Now all that remains to be shown is that once $\xi_k$ is sufficiently close to either of the stationary distributions $\xi^0$ or $\xi^1$, the tree-based transport inequalities $\mathcal{I}(G,T,\mathcal{O},\xi_k)$ will be satisfied. To prove this, we will first show that $\xi^0$ and $\xi^1$ lie on the interior of the region of distributions that satisfy the inequalities. The next lemma helps show that $\xi^0$ and $\xi^1$ satisfy the inequalities.

\begin{lemma} \label{lem: algorithm nice on stat dib}
Suppose we have a bipartite graph $G$ with sides $S_0$ and $S_1$ and a distribution $\xi$ such that for $w \in S_0$ we have $\xi(w) = \frac{\deg(w)}{|E(G)|}$ and for $w \in S_1$ we have $\xi(w) = -\frac{\deg(w)}{|E(G)|}$. Then pick an arbitrary spanning tree T and $r$-monotone ordering $\mathcal{O}$ on $V(G)$. Consider the tree-based transport plan $A(G,T,\mathcal{O},\xi)$. After each step, for each $i$ for $i \leq n-2$, we have that $A_i(\xi)(w_j) \geq \frac{1}{|E(G)|}$ for $w_j\in S_0$ with $j > i$ and that $A_i(\xi)(w_j) \leq -\frac{1}{|E(G)|}$ for $w_j\in S_1$ with $j > i$.
\end{lemma}

\begin{proof}
We know by Lemma \ref{lem: linear} that for all $w \in G$ and for all $0 \leq i \leq n-2$, we have $A_i(|E(G)|\xi)(w) = |E(G)|((A_i(\xi)(w))$. Thus, it suffices to show that after all steps $i$ for $i \leq n-2$, we have that $A_i(|E(G)|\xi)(w_j) \geq 1$ for $w_j\in S_0$ with $j > i$ and that $A_i(|E(G)|\xi)(w_j) \leq -1$ for $w_j\in S_1$ with $j > i$.

To prove this, for all $0 \leq i \leq n-2$, we define the graph $G_i$ to consist of the vertex set $V(G_i) = \{w_{i+1},\ldots w_n\}$ and all the edges of $E(G)$ that have both endpoints in $V(G_i)$. It suffices to show by induction on $i$ that for $i \leq n-2$, we have $A_i(|E(G)|\xi)(w_j) = \deg_{G_i}w_j$ for $w_j\in S_0$ with $j > i$ and we have $A_i(|E(G)|\xi)(w_j) = -\deg_{G_i}w_j$ for $w_j\in S_1$ with $j > i$.

Base case: When $i=0$, we note that $G_0 = G$. When $i = 0$, by the definition of $\xi$, we have that $A_i(|E(G)|\xi)(w_j) = \deg_{G}w_j$ for $w_j\in S_0$ with $j > i$ and that $A_i(|E(G)|\xi)(w_j) = -\deg_{G}w_j$ for $w_j\in S_1$ with $j > i$.

Inductive step: The inductive hypothesis is that $A_i(|E(G)|\xi)(w_j) = \deg_{G_i}w_j$ for $w_j\in S_0$ with $j > i$ and that $A_i(|E(G)|\xi)(w_j) = -\deg_{G_i}w_j$ for $w_j\in S_1$ with $j > i$. Given that this is true for $i-1$, we want to show that it is true for $i$. 

We suppose that $w_i\in S_0$; the case where $w_i\in S_1$ will proceed analogously. After $i-1$ steps, $w_i$ has a mass of $\deg_{G_{i-1}}w_i$. During the $i$th step, this mass is distributed evenly among $w_j\sim w_i$ with $j > i$; we note that there are exactly $\deg_{G_{i-1}}w_i$ of these neighbors. Thus, each $w_j$ will receive $+1$ mass. By the inductive hypothesis we have that before step $i$, each of these neighbors $w_j$ had $-\deg_{G_{i-1}}w_j$ mass, since each of the neighbors of $w_i$ is in $S_1$, the opposite side of the bipartite graph. Then, after step $i$, each $w_j$ has mass $-(\deg_{G_{i-1}}w_j - 1)$, and the remaining vertices with indices greater than $i$ have the same mass as before. We note that for all $\ell > i$, if $w_{\ell} \sim w_i$ then $\deg_{G_i}w_\ell = \deg_{G_{i-1}}w_\ell - 1$ because the edge $\{(w_\ell, w_i\}$ is being removed, and otherwise $\deg_{G_i}w_\ell = \deg_{G_{i-1}}w_\ell$. We have just shown that this is exactly the mass at all vertices with indices greater than $i$ after the $i$th step of the algorithm. Thus, at each vertex $w_\ell$ with $\ell > i$, we have that for $w_\ell\in S_0$, the mass at $w_\ell$ after $i$ steps is $\deg_{G_i}w_\ell$ and for $w_\ell\in S_1$, the mass at $w_\ell$ after $i$ steps is $-\deg_{G_i}w_\ell$. We proceed analogously in the case where $w_i\in S_1$. This proves the inductive hypothesis, and therefore proves the lemma.
\end{proof}

We are now ready to show that $\xi^0$ and $\xi^1$ lie on the interior of the region of distributions that satisfy the inequalities.

\begin{corollary}\label{cor: xi interior}
For any Guvab $\mathcal{G}$ where $W = 1$ and $\beta < 1$, we have that $\xi^0$ and $\xi^1$ lie strictly on the interior of the region $R \subset \mathbb{R}^{|V(G)|}$ of distributions $\xi$ that satisfy the tree-based transport inequalities $\mathcal{I}(G,T,\mathcal{O},\xi)$.
\end{corollary}

\begin{proof}
We prove this for $\xi^0$; by symmetry it will hold for $\xi^1$ as well since $\xi^1 = -\xi^0$. If the sides of $G$ are $S_0$ and $S_1$, with $u\in S_0$ and $v\in S_1$, then for $w\in S_0$, we have that $$\xi^0(w) = \lim_{k\to\infty} \mu_{2k}(w) - \lim_{k\to\infty} \nu_{2k}(w) = \frac{\deg(w)}{|E(G)|} - 0 = \frac{\deg(w)}{|E(G)|}$$ and for $w\in S_1$ we have that $$\xi^0(w) = \lim_{k\to\infty} \mu_{2k+1}(w) - \lim_{k\to\infty} \nu_{2k+1}(w) = 0 - \frac{\deg(w)}{|E(G)|} = - \frac{\deg(w)}{|E(G)|}.$$ Then for all $t, w \in G$ such that $t \sim w$, we have that $\xi(t)\xi(w) < 0$. We also have that by Lemma~\ref{lem: algorithm nice on stat dib}, $\xi^0(w_j)A_i(\xi^0)(w_j) \geq \frac{1}{|E(G)|^2} > 0$ for all $0 \leq i \leq |V(G)|- 2$ and $i < j \leq V(G)$. Thus, $\xi^0$ and $\xi^1$ lie strictly on the interior of the region $R \subset \mathbb{R}^{|V(G)|}$ of distributions $\xi$ that satisfy the tree-based transport inequalities $\mathcal{I}(G,T,\mathcal{O},\xi)$.
\end{proof}

Using these results, we are now ready to prove the main claim that the Wasserstein distance is eventually constant when $W = 1$ and $\beta < 1$.

We first define a variable that corresponds to how long $\{W_k\}$ takes to reach constancy. Note that this variable can be infinity if $\{W_k\}$ is not eventually constant.

\begin{definition} \label{rho}
For any Guvab $\mathcal{G}$ where $W_k \to 1$, define $\rho(\mathcal{G})$ to be $$\inf \{N \in \mathbb{Z} : \{W(\mu_k,\nu_k)\}_{k \geq N} = (1,1,1,\ldots) \}.$$
\end{definition}

\begin{theorem}\label{thm:winnie the when it's constant for W = 1 fairy}
For any Guvab $\mathcal{G}$ with $W = 1$ and $\beta < 1$, we have $\rho(\mathcal{G}) < \infty$. 
\end{theorem}

\begin{proof}
Pick an arbitrary spanning tree $T$ of $G$ and $r$-monotone ordering $\mathcal{O}$. By Corollary \ref{cor: xi interior}, $\xi^0$ and $\xi^1$ are on the interior of the region $R \subset \mathbb{R}^{|V(G)|}$ of distributions $\xi$ that satisfy the tree-based transport inequalities $\mathcal{I}(G,T,\mathcal{O},\xi)$. We note that all the inequalities in $\mathcal{I}(G,T,\mathcal{O},\xi)$ can be written in the form $f(\xi) > 0$, where $f: \mathbb{R}^{|V(G)|} \to \mathbb{R}$ is a continuous function. Thus, by the definition of a continuous function, there exists some $\varepsilon > 0$ such that for all $\xi \in \mathbb{R}^{|V(G)|}$ that satisfy $|\xi(w) - \xi^0(w)| < \varepsilon$ for all $w \in G$ or satisfy $|\xi(w) - \xi^1(w)| < \varepsilon$ for all $w \in G$, we have that $\xi \in R$. We also know, by the formal definition of a limit, that there exists some $N$ such that for all $k \geq N$ and all $w \in G$, we have $|\xi_{2k}(w) - \xi^0(w)| < \varepsilon$ and $|\xi_{2k+1}(w) - \xi^1(w)| < \varepsilon$. Thus, for all $k \geq 2N$, we have $\xi_k \in R$. By Corollary~\ref{cor: when ineqs, dist is 1}, for all $k \geq 2N$, we have that $W_k = 1$. Hence $\rho(\mathcal{G}) \leq 2N < \infty$.
\end{proof}

We next hope to characterize how long it takes the Wasserstein distance of these Guvabs with $W=1$ and $\beta < 1$ to become constant. In particular, we prove upper and lower bounds for $\rho(\mathcal{G})$. We start with the upper bound. To prove this upper bound, we first prove a lemma quantifying exactly how close to $\xi^0$ or $\xi^1$ a distribution must be in order for the tree-based transport inequalities $\mathcal{I}(G,T,\mathcal{O},\xi)$ to be satisfied.

\begin{lemma} \label{lem: explicit epsilon bounds}
Consider a Guvab with $W = 1$ and $\beta < 1$. Pick an arbitrary spanning tree $T$ and $r$-monotone ordering $\mathcal{O}$. Let $\varepsilon(G) = \frac{1}{|V||E|}$. If for a distribution $\xi$ it is true that for all vertices $w$ we have that $|\xi(w) - \xi^0(w)| < \varepsilon(G)$ or it is true that for all vertices $w$ we have that $|\xi(w) - \xi^1(w)| < \varepsilon(G)$, then $\xi$ satisfies the tree-based transport inequalities $\mathcal{I}(G,T,\mathcal{O},\xi)$.
\end{lemma}

\begin{proof}
We prove this for $\xi^0$, and an analogous argument will hold for $\xi^1$.

We note that by Lemma \ref{lem: algorithm nice on stat dib} we have that if we start with $\xi^0$, then at any point in the tree-based transport algorithm through step $|V| - 2$, the absolute value of the mass at any vertex is at least $\frac{1}{|E|}$. Thus, if at any point $i$ in the algorithm through step $|V| - 2$ the mass at a vertex differs by at most $\frac{1}{|E|}$ from $A_i(\xi^0)$, then the tree-based transport inequalities $\mathcal{I}$ are satisfied because mass is never the wrong sign. 

It thus suffices to show that for all $0 \leq i \leq |V| - 2$ and for all $w \in G$, we have $|A_i(\xi)(w) - A_i(\xi^0)(w)| \leq \frac{1}{|E|}$. To prove this, we note that $\xi(w) - \xi^0(w)$ is a zero-sum distribution, and by Lemma \ref{lem: linear} for all $i$ and for all $w$, we have $A_i(\xi)(w) = A_i(\xi^0)(w) + A_i(\xi - \xi^0)(w)$. 
We consider the quantity $\sum_{w \in G} |A_i(\xi)(w) - A_i(\xi^0)(w)| = \sum_{w\in G} |A_i(\xi -\xi^0)(w)|$. This will be nonincreasing as $i$ gets larger, since at step $i$ of the algorithm the absolute value of the mass at $w_i$ decreases by exactly $|A_{i-1}(\xi-\xi^0)(w_i)|$ while the sum of absolute values at $w_i$'s neighbors cannot increase by more than $|A_{i-1}(\xi-\xi^0)(w_i)|$. The maximum value of this sum is $|V|\varepsilon(G)$ (since this is an upper bound for the value at the beginning). We know that $\max_{w \in G} |A_i(\xi)(w) - A_i(\xi^0)(w)| \leq \sum_{w \in G} |A_i(\xi)(w) - A_i(\xi^0)(w)|$ so $\max_{w \in G} |A_i(\xi)(w) - A_i(\xi^0)(w)| \leq |V|\varepsilon(G) = \frac{1}{|E|}$, which is exactly what we wanted to show, so we are done.
\end{proof}

With this lemma established, we can now prove our upper bound for $\rho(\mathcal{G})$.

\begin{lemma}
Let $\lambda_{\max}$ be $\displaystyle \max_{|\lambda| \in L: |\lambda| < 1} |\lambda|$ where L is the set of all eigenvalues of $X$ and $Y$. Then for a Guvab $\mathcal{G}$ where $W = 1$ and $\beta < 1$, we have $\displaystyle \rho(\mathcal{G}) \leq \frac{10\ln|V|}{1-\lambda_{\max}^2}$.
\end{lemma}

\begin{proof}
We use \cite[Prop. 3]{diaconis1991geometric}.
The Markov chains $X_{2k}$ and $Y_{2k}$ are both converging to their even stationary distributions $\lim_{k\to\infty} \mu_{2k}$ and $\lim_{k\to\infty} \nu_{2k}$. For convenience, denote $\lim_{k\to\infty} \mu_{2k}$ by $\gamma_u$ and $\lim_{k\to\infty} \nu_{2k}$ by $\gamma_v$. Once at all vertices, $\mu_{2k}$ and $\nu_{2k}$ are both less than or equal to $\frac{1}{2|V||E|}$ away from their respective stationary distributions, $\xi_{2k}$ will satisfy the tree-based transport inequalities $\mathcal{I}(G,T,\mathcal{O},\xi_{2k})$ by Lemma \ref{lem: explicit epsilon bounds}. Since $X_{2k}$ and $Y_{2k}$ are both Markov chains with limiting distributions, we use notation analogous to that of \cite{Sinclair92improvedbounds} and say that $\Delta_{\even\,u}(k) = \displaystyle \frac{1}{2}\sum_{w\in G}|\mu_{2k}(w) - \gamma_u(w)|$. Similarly, $\Delta_{\even\,v}(k) = \displaystyle \frac{1}{2}\sum_{w\in G}|\nu_{2k}(w) - \gamma_v(w)|$. Then for $\varepsilon > 0$ and $x \in \{u,v\}$ we let $\tau_{\even \, x}(\varepsilon)$ be the minimum nonnegative integer $k$ such that $\Delta_{\even\,x}(k') \leq \varepsilon$ for all $k' \geq k$. Thus, by \cite[Prop. 3]{diaconis1991geometric}, the time $\rho_{\even}$ it takes for $W_{2k}$ to eventually have distance $1$ satisfies $$\rho_{\even}\leq\max_{x \in \{u,v\}} 2\tau_{\even \,x}(\frac{1}{2|V||E|}) \leq \max_{x \in \{u,v\}} \frac{2}{1-\lambda_{\max}^2}(\ln{\frac{1}{\gamma_x(x)}} + \ln{\cfrac{1}{\frac{1}{2|V||E|}}}).$$ 
Then we just need to bound the right-hand side above. This gives 
\begin{align*}
    \displaystyle \rho_{\even}&\leq\frac{2}{1-\lambda_{\max}^2}(\ln{|E|} + \ln{2|V||E|}) 
    =\frac{2}{1-\lambda_{\max}^2}(\ln{2|V||E|^2})\\
    &\leq \frac{2}{1-\lambda_{\max}^2}(\ln{|V|^3(|V|-1)^2})\\  
    &< \frac{2}{1-\lambda_{\max}^2}\cdot 5 \ln(|V|)
    = \frac{10\ln|V|}{1-\lambda_{\max}^2}.
\end{align*}

By similar reasoning, the same bound works for $\rho_{\odd}$, 
the time it takes for $W_{2k+1}$ to eventually have distance $1$. Thus, $\displaystyle\frac{10\ln|V|}{1-\lambda_{\max}^2}$ is an upper bound for $\rho(\mathcal{G})$.
\end{proof}

We now establish a lower bound for $\rho(\mathcal{G})$.
\begin{lemma}
For a Guvab $\mathcal{G}$ where $W = 1$ and $\beta < 1$, we have $\displaystyle \rho(\mathcal{G}) \geq \frac{\emph{\dist(u,v)}}{2} - 1$.
\end{lemma}

\begin{proof}
We note that $\mu_k(t) = 0$ for $t \in V$ if $\dist(t,u) > k$. Similarly, $\nu_k(w) = 0$ for $w \in V$ if $\dist(w,v) > k$. Suppose $k < \frac{\dist(u,v)}{2} - 1$ and consider any pair of vertices $t,w$ such that $\mu_k(t) > 0$ and $\nu_k(w) > 0$. Then $\dist(t,u) \leq \frac{\dist(u,v)}{2} - 1$ and $\dist(w,v) \leq \frac{\dist(u,v)}{2} - 1$, so $\dist(t,w) \geq \dist(u,v) - (\dist(t,u) +\dist(w,v)) = 2$. Therefore all mass will have to move a distance of at least 2 to get from $\mu_k$ to $\nu_k$, so $W_k \geq 2 > 1$. 
\end{proof}

\section{Convergence when $W = \frac{1}{2}$}
In this section, we consider Guvabs where $W = \frac{1}{2}$ and $\beta < 1$. Recall that these are exactly the Guvabs for which $G$ is bipartite and $0 = \alpha < \beta < 1$. As in the previous section, and for similar reasons, the Wasserstein distance will eventually be the sum of positive mass. In this case, however, the Wasserstein distance is not eventually constant but rather an exponential that we can express explicitly. To prove this, we proceed by a similar strategy as in the $W = 1$ case. In particular, we show that the tree-based transport inequalities will eventually be satisfied, and compute the Wasserstein distance when these inequalities are satisfied.

In the next three results, we show that the tree-based transport inequalities will eventually be satisfied, and provide an initial expression for what the Wasserstein distance will be when the tree-based transport inequalities are satisfied. Later, we will calculate exactly what this expression for the Wasserstein distance evaluates to.

We begin by showing in the next two results that, analogously to before, $\xi^0$ and $\xi^1$ lie on the interior of the region of distributions that satisfy the inequalities.

\begin{lemma} \label{lem: for 1/2 algorithm nice on stat dib}
Suppose we have a bipartite graph $G$ with sides $S_0$ and $S_1$ and a distribution $\xi$ such that for $w\in S_0$ $\xi(w) = \frac{\deg(w)}{2|E(G)|}$ and for $w\in S_1$ $\xi(w) = -\frac{\deg(w)}{2|E(G)|}$. Then pick an arbitrary spanning tree T and $r$-monotone ordering $\mathcal{O}$ on $V(G)$. Consider the tree-based transport plan $A(G,T,\mathcal{O},\xi)$. We have that after each step $i$ for $i \leq n-2$, for $w_j\in S_0$ with $j > i$, we have that $A_i(\xi)(w_j) \geq \frac{1}{2|E(G)|}$ and for $w_j\in S_1$ with $j > i$ we have that $A_i(\xi)(w_j) \leq -\frac{1}{2|E(G)|}$.
\end{lemma}

\begin{proof}
We note that this is nearly the same as Lemma~\ref{lem: algorithm nice on stat dib}, but differs by a constant factor of $\frac{1}{2}$. Given the distribution $\xi$, we know by Lemma~\ref{lem: algorithm nice on stat dib} that after all steps $i$ for $i \leq n-2$, for $w_j\in S_0$ with $j>i$, we have that $A_i(2\xi)(w_j) \geq \frac{1}{|E(G)|}$ and for $w_j\in S_1$ with $j>i$ we have that $A_i(2\xi)(w_j) \leq -\frac{1}{|E(G)|}$. We also know that $A_i(2\xi) = 2A_i(\xi)$ by Lemma \ref{lem: linear} so this means for $w_j\in S_0$ with $j>i$, we have that $2A_i(\xi)(w_j) \geq \frac{1}{|E(G)|}$ and for $w\in S_1$ with $j>i$, we have that $2A_i(\xi)(w_j) \leq -\frac{1}{|E(G)|}$. Dividing both sides by 2, we get that after all steps $i$ for $i \leq n-2$, for $w_j\in S_0$ with $j>i$, we have that $A_i(\xi)(w_j) \geq \frac{1}{2|E(G)|}$ and for $w_j\in S_1$ with $j > i$, we have that $A_i(\xi)(w_j) \leq -\frac{1}{2|E(G)|}$. 
\end{proof}

\begin{corollary} \label{cor: stat dib on interior for w=1/2}
For any Guvab $\mathcal{G}$ where $W = \frac{1}{2}$ and $\beta < 1$, we have that $\xi^0$ and $\xi^1$ lie strictly on the interior of the region $R \subset \mathbb{R}^{|V(G)|}$ of distributions $\xi$ that satisfy the tree-based transport inequalities $\mathcal{I}(G,T,\mathcal{O},\xi)$.
\end{corollary}
\begin{proof}
We note that when $W = \frac{1}{2}$ and $\beta < 1$, we have that $\beta > 0$ so for all $w \in G$, we have that $\lim_{k\to\infty} \nu_k(w) = \frac{\deg(w)}{2|E(G)|}$. We also know that if $G$ has sides $S_0$ and $S_1$ with $u\in S_0$, for all $w \in S_0$ we have that $\displaystyle \lim_{k\to\infty} \mu_{2k}(w) = \frac{\deg(w)}{|E(G)|}$ and for all $w \in S_1$ we have that $\displaystyle \lim_{k\to\infty} \mu_{2k}(w) = 0$. Similarly, for all $w \in S_1$ we have that $\displaystyle \lim_{k\to\infty} \mu_{2k+1}(w) = \frac{\deg(w)}{|E(G)|}$ and for all $w \in S_0$ we have that $\displaystyle \lim_{k\to\infty} \mu_{2k+1}(w) = 0$. Thus for all $w \in S_0$ we have that $\xi^0(w) = \frac{\deg(w)}{2|E(G)|}$ and for all $w \in S_1$ we have that $\xi^0(w) = \frac{-\deg(w)}{2|E(G)|}$. Also $\xi^1 = -\xi^0$.

We prove the claim for $\xi^0$ - by symmetry it will hold for $\xi^1$ as well since $\xi^1 = -\xi^0$. If the sides of $G$ are $S_0$ and $S_1$, then for $w\in S_0$ $\xi^0(w) = \frac{\deg(w)}{2|E(G)|}$ and for $w\in S_1$ $\xi^0(w) = -\frac{\deg(w)}{2|E(G)|}$. Then we have that for all $t, w \in G$ such that $t \sim w$, the product $\xi(t)\xi(w) < 0$. We also have that by Lemma~\ref{lem: for 1/2 algorithm nice on stat dib}, $\xi^0(w)A_i(\xi^0)(w) \geq \frac{1}{4|E(G)|^2} > 0$ holds for all $w \in G$ and for all $0 \leq i \leq |V(G)|- 2$.
\end{proof}

We now know that $\xi^0$ and $\xi^1$ are on the interior of the region satisfying the inequalities. We can hence proceed similarly to section 4 to show that $\xi_k$ will eventually satisfy the inequalities and thus $W_k$ will be the sum of positive mass.

\begin{corollary} \label{cor: for 1/2, dist is eventually sum of pos mass}
For any Guvab where $W = \frac{1}{2}$ and $\beta < 1$, there exists $N$ such that for all $k \geq N$, 
$$W_k = \frac{1}{2}\sum_{w \in G} |\xi_k(w)|.$$
\end{corollary}
\begin{proof}
We know by Corollary \ref{cor: stat dib on interior for w=1/2} that $\xi^0$ and $\xi^1$ lie on the interior of $R$. Therefore, as in the proof of Theorem \ref{thm:winnie the when it's constant for W = 1 fairy}, by the formal definition of a limit there exists some $N$ such that for all $k \geq N$, we have that $\xi_k \in R$ and thus $\mathcal{I}(G,T,\mathcal{O},\xi_k)$ are satisfied. We note that Corollary \ref{cor: when ineqs, dist is sum of pos mass} holds for any Guvab, including the ones we are currently inspecting, so if the tree-based transport inequalities $\mathcal{I}(G,T,\mathcal{O},\xi_k)$ are satisfied, $W_k = \frac{1}{2}\sum_{w \in G} |\xi_k(w)|$. Thus for all $k \geq N$, we have that $W_k = \frac{1}{2}\sum_{w \in G} |\xi_k(w)|.$
\end{proof}

We now know that eventually, the Wasserstein distance will be the sum of positive mass, so it remains to calculate the sum of positive mass. To do this, we will first need to define an auxiliary Markov chain and prove some properties of this Markov chain.

\begin{definition}
Let $s(\alpha)$ be a two-state Markov chain with states $s_0$ and $s_1$, where we start at $s_0$, and at all times we have an $\alpha$ chance of staying at our current state and a $1-\alpha$ chance of switching to the other state. Then define $(\sigma_\alpha)_k$ to be the probability distribution after $k$ steps of this Markov chain.
\end{definition}

\begin{lemma} \label{lem: mark the markov chain fairy}
For the Markov chain defined above, $(\sigma_\alpha)_k(s_0) = 0.5 + 0.5(2\alpha-1)^k$ and $(\sigma_\alpha)_k(s_1) = 0.5-0.5(2\alpha-1)^k$.
\end{lemma}

\begin{proof}
We will proceed by induction on $k$, using the transition probabilities to go from $(\sigma_\alpha)_k$ to $(\sigma_\alpha)_{k+1}$.

Base case: When $k=0$, we know that, since the Markov chain starts at $s_0$, we have $(\sigma_\alpha)_0(s_0) = 1 = 0.5+0.5(2\alpha-1)^0$ and $(\sigma_\alpha)_0(s_1) = 0 = 0.5-0.5(2\alpha-1)^0$. 

Inductive step: Suppose $(\sigma_\alpha)_k(s_0) = 0.5+0.5(2\alpha-1)^k$ and $(\sigma_\alpha)_k(s_1) = 0.5-0.5(2\alpha-1)^k$. We know that 
\begin{align*}
(\sigma_\alpha)_{k+1}(s_0) &= \alpha(\sigma_\alpha)_k(s_0) + (1-\alpha)(\sigma_\alpha)_k(s_1)\\ &= \alpha(0.5+0.5(2\alpha-1)^k) + (1-\alpha)(0.5-0.5(2\alpha-1)^k)\\ &= 0.5 + (2\alpha-1)\cdot 0.5(2\alpha-1)^k\\ &= 0.5 + 0.5(2\alpha-1)^{k+1}. 
\end{align*}
Similarly,
\begin{align*}
(\sigma_\alpha)_{k+1}(s_1) &= (1-\alpha)(\sigma_\alpha)_k(s_0) + \alpha(\sigma_\alpha)_k(s_1)\\ &= (1-\alpha)(0.5+0.5(2\alpha-1)^k) + \alpha(0.5-0.5(2\alpha-1)^k)\\ &= 0.5+(2\alpha-1)\cdot (-0.5)(2\alpha-1)^k\\ &= 0.5 - 0.5(2\alpha-1)^{k+1}.
\end{align*}
\end{proof}

We now have all the tools we need to explicitly calculate the sum of positive mass. The next lemma tells us what the sum of positive mass will be.

\begin{lemma} \label{lem: sum of pos mass for 1/2}
When $\beta < 1$ and $W = \frac{1}{2}$, there exists some $N$ such that for all $k \geq N$, we either have that $$\frac{1}{2}\sum_{w \in G} |\xi_k(w)| = 0.5 + 0.5(1-2\beta)^k$$ or that $$\frac{1}{2}\sum_{w \in G} |\xi_k(w)| = 0.5 - 0.5(1-2\beta)^k.$$
\end{lemma}

\begin{proof}
We know $G$ is bipartite; say it has sides $S_0$ and $S_1$. We know $0 = \alpha < \beta < 1$. Assume without loss of generality that $v \in S_0$. If $u$ is on side $S_0$, then eventually for $w\in S_0$, we have that $\xi_{2k}(w)$ gets arbitrarily close to $\displaystyle\frac{\deg(w)}{2|E(G)|}$ and for $w\in S_1$, we have that $\xi_{2k}(w)$ gets arbitrarily close to $\displaystyle-\frac{\deg(w)}{2|E(G)|}$. In particular, for some $N$, for all $k \geq N$ we have $\xi_{2k}(w) > 0$ if and only if $w\in S_0$. Then for some $N$, for all $k \geq N$, when $u\in S_0$, the total positive mass of $\xi_{2k}$ is $\displaystyle\sum_{w \in S_0} \xi_{2k}(w)$. Similarly, for some $N$, for all $k \geq N$ we have $\xi_{2k+1}(w) > 0$ if and only if $w\in S_1$. Thus, for some $N$, for all $k \geq N$, when $u\in S_0$, the total positive mass of $\xi_{2k+1}$ is $\displaystyle\sum_{w \in S_1} \xi_{2k+1}(w)$. 

By an analogous argument, when $u\in S_1$, for some $N$, for all $k \geq N$, the total positive mass of $\xi_{2k}$ is $\displaystyle\sum_{w \in S_0} \xi_{2k}(w)$. Similarly, for some $N$, for all $k \geq N$, the total positive mass of $\xi_{2k+1}$ is $\displaystyle\sum_{w \in S_1} \xi_{2k+1}(w)$.

Thus, to calculate what the sum of positive mass eventually equals, we simply consider how much mass of $\mu$ and $\nu$ is on each side of the bipartite graph so that we know how much mass of $\xi$ is on each side. We note that for any random walk with laziness $\beta$, at all steps the mass on a given side has a probability $\beta$ of staying on that side and a probability $1-\beta$ of moving to the other side, since any mass that moves along an edge moves to the other side. Thus the mass of $\nu_k$ on $S_0$ and $S_1$ behaves identically to the mass of $s(\beta)$ on $s_0$ and $s_1$. In other words, the amount of mass of $\nu_k$ on $S_0$ is $(\sigma_\beta)_k(s_0) = 0.5 + 0.5(2\beta-1)^k$ and the amount of mass of $\nu_k$ on $S_1$ is $(\sigma_\beta)_k(s_1) = 0.5 - 0.5(2\beta-1)^k$ by Lemma \ref{lem: mark the markov chain fairy}. Similarly, if $u\in S_0$, then the amount of mass of $\mu_k$ on $S_0$ is $(\sigma_0)_k(s_0) = 0.5 + 0.5(-1)^k$ and the amount of mass of $\mu_k$ on $S_1$ is $(\sigma_0)_k(s_1) = 0.5 - 0.5(-1)^k$. By symmetry, if $u\in S_1$, then the amount of mass of $\mu_k$ on $S_0$ is $(\sigma_0)_k(s_1) = 0.5 - 0.5(-1)^k$ and the amount of mass of $\mu_k$ on $S_1$ is $(\sigma_0)_k(s_0) = 0.5 + 0.5(-1)^k$.

This means that if $u\in S_0$, 
\begin{align*}
\sum_{w \in S_0} \xi_{k}(w) &= \sum_{w \in S_0} \mu_{k}(w) - \sum_{w \in S_0} \nu_{k}(w)\\ &= (\sigma_0)_k(s_0) - (\sigma_\beta)_k(s_0)\\ &= 0.5 + 0.5(-1)^k - (0.5 + 0.5(2\beta-1)^k)
\end{align*}
and 
\begin{align*}
\sum_{w \in S_1} \xi_{k}(w) &= \sum_{w \in S_1} \mu_{k}(w) - \sum_{w \in S_1} \nu_{k}(w)\\ &= (\sigma_0)_k(s_1) - (\sigma_\beta)_k(s_1)\\ &= 0.5 - 0.5(-1)^k - (0.5 - 0.5(2\beta-1)^k).
\end{align*}
Then the total positive mass of $\xi_{2k}$ is $$\sum_{w \in S_0} \xi_{2k}(w) = 0.5 + 0.5(-1)^{2k} - (0.5 + 0.5(2\beta-1)^{2k}) = 0.5 - 0.5(1-2\beta)^{2k}$$ and the total positive mass of $\xi_{2k+1}$ is $$\sum_{w \in S_1} \xi_{2k+1}(w) = 0.5 - 0.5(-1)^{2k+1} - (0.5 - 0.5(2\beta-1)^{2k+1}) = 0.5 - 0.5(1-2\beta)^{2k+1}.$$ Thus, for some $N$, the sum of the positive mass of $\xi_k$ is $0.5 - 0.5(1-2\beta)^k$ for all $k \geq N$. 

If $u\in S_1$, then we have that $\displaystyle\sum_{w \in S_0} \xi_{k}(w)= (\sigma_0)_k(s_1) - (\sigma_\beta)_k(s_0)$ and we have that $\displaystyle\sum_{w \in S_1} \xi_{k}(w) = (\sigma_0)_k(s_0) - (\sigma_\beta)_k(s_1)$. By calculating this out analogously to above, we see that if $u\in S_1$ there exists some $N$ such that the sum of positive mass of $\xi_k$ is $0.5 + 0.5(1-2\beta)^k$ for all $k \geq N$.
\end{proof}

We now know that the Wasserstein distance will be the sum of positive mass, and we know exactly what the sum of positive mass will eventually be. Thus, we know exactly what the Wasserstein distance will eventually be. The next theorem therefore states explicitly the rate of convergence of the Wasserstein distance when $\beta < 1$ and $W = \frac{1}{2}$.

\begin{theorem} \label{thm: washington the w=1/2 convergence theorem}
For any Guvab where $W = \frac{1}{2}$ and $\beta < 1$, for some $N$ it will be true that for all $k \geq N$, we have that $|W_k - \frac{1}{2}| = 0.5|1-2\beta|^k$.
\end{theorem}

\begin{proof}
Corollary \ref{cor: for 1/2, dist is eventually sum of pos mass} tells us that for some $N_1$, we will have $W_k = \frac{1}{2}\sum_{w\in G} |\xi_k(w)|$ for all $k \geq N_1$. Lemma \ref{lem: sum of pos mass for 1/2} tells us that for some $N_2$, we will have $\frac{1}{2}\sum_{w\in G} |\xi_k(w)| = 0.5 + 0.5(1-2\beta)^k$ for all $k \geq N_2$ or we will have $\frac{1}{2}\sum_{w\in G} |\xi_k(w)| = 0.5 - 0.5(1-2\beta)^k$ for all $k \geq N_2$. This means that for all $k \geq N_2$, we have $|(\frac{1}{2}\sum_{w\in G} |\xi_k(w)|) - \frac{1}{2}| = 0.5|1-2\beta|^k$. Thus, for all $k \geq \max(N_1,N_2)$, we have $$\left|W_k - \frac{1}{2}\right| = \left|\left(\frac{1}{2}\sum_{w\in G} |\xi_k(w)|\right) - \frac{1}{2}\right| = 0.5|1-2\beta|^k.$$
\end{proof}

Finally, we want to characterize when the Wasserstein distance is eventually constant when $W = \frac{1}{2}$ and $\beta < 1$. This will fit into our larger characterization of eventual constancy for all Guvabs with $\beta < 1$.

\begin{corollary}\label{cor: wendy the when it's constant for W = 1/2 fairy}
When $W = \frac{1}{2}$ and $\beta < 1$, we have that $\rho < \infty$ if and only if $\beta = \frac{1}{2}$.
\end{corollary}
\begin{proof}
This follows directly from Theorem \ref{thm: washington the w=1/2 convergence theorem}.
\end{proof}

\section{Convergence when $W = 0$}
In this section we consider the case of Guvabs where $W = 0$ and $\beta < 1$. Recall that these are exactly the Guvabs enumerated in Theorem \ref{thm: 0-convergence} for which $\beta < 1$. We start by showing that the rate of convergence of $\{W_{2k}\}$ is exponential when it is not eventually constant. By an analogous argument, the rate of convergence of $\{W_{2k+1}\}$ is exponential when it is not eventually constant. We will then investigate exactly when $\{W_k\}$ is eventually constant.

Theorem \ref{thm: Simba the sim-an-exponential when W=0 fairy} states that unless it is eventually constant, the rate of convergence of $\{W_{2k}\}$ is exponential, and in particular $W_{2k} \sim c\cdot \lambda_{\even}^{2k}$. We go about proving this by showing in the next two lemmas that $W_{2k}$ must be one of finitely many expressions, all of which are approximately some exponential.

The next lemma shows that $W_{2k}$ must be one of finitely many expressions.

\begin{lemma}\label{lem: finite function set}
For any Guvab $\mathcal{G}$, there exists a finite set $F = \{f_i,f_2,\ldots, f_m\}$ of 1-Lipschitz functions $f_i: V(G) \to \mathbb{R}$ such that for all $k$ there exists $f \in F$ such that $\displaystyle W_k = \sum_{w\in G} f(w)\xi_k(w).$
\end{lemma}

\begin{proof}
We consider the set $L$ of possible 1-Lipschitz functions $\ell$ on the graph $G$ such that $\sum_{w \in V(G)} \ell(w) = 0$ (any other 1-Lipschitz function can be transformed into such a 1-Lipschitz function by adding some value to all entries). The criteria for a function $\ell$ to be a 1-Lipschitz function are that for each pair of vertices $w_1$ and $w_2$ we have that $\ell(w_1) - \ell(w_2) \leq d(w_1,w_2)$ and $\ell(w_2) - \ell(w_1) \leq d(w_1,w_2)$. We also have that $\sum_{w\in G} \ell(w) = 0$. These each form hyperplanes in $\mathbb R^{|V(G)|}$. Additionally, from these criteria we know that none of the entries of $\ell$ can be more than $|V(G)|$ because then since the max distance between any two vertices is $|V(G)|$ there would be no negative entries. Thus, the set of 1-Lipschitz functions forms a closed set bounded by a polytope in $\mathbb R^{|V(G)|}$.  
For any cost function $C$ on the graph $G$, we have that $\sum_{w \in V(G)} C(w)\ell(w)$ is a linear function on $L$. Thus $\displaystyle \argmax_{\ell \in L} \sum_{w \in V(G)} C(w)\ell(w)$ is one of the corners of the polytope. There are finitely many of these corners, corresponding to finitely many 1-Lipschitz functions $\{f_1,f_2,\ldots, f_m\}$. We also know that $W_k = \max_{\ell \in L} \sum_{w \in V(G)} \xi_k(w)\ell(w)$, so it is thus maximizing the cost function $\xi$, and thus for all $k$ there exists $f \in F$ such that $\displaystyle W_k = \sum_{w\in G} f(w)\xi_k(w)$.
\end{proof}

We now know that $W_{2k}$ will be one of finitely many expressions. The next lemma shows that each of these expressions is approximately exponential.

\begin{lemma} \label{lem: 1-lip are exponential}
For any Guvab $\mathcal{G}$ for which $W=0$ and any 1-Lipschitz function $f$, there exists some $0 < \lambda_f < 1$ and some constant $c_f$ such that $$\sum_{w\in G} f(w)\xi_{2k}(w) \sim c_f \cdot \lambda_f^{2k}$$ unless there exists some $N$ such that for all $k > N$ we have that $\sum_{w\in G} f(w)\xi_{2k}(w) = 0$.
\end{lemma}
\begin{proof}
Assume that there does not exist any $N$ such that for all $k > N$ we have that $\sum_{w\in G} f(w)\xi_{2k}(w) = 0$.

We know by Lemma \ref{lem: eileen the eigval sum fairy} that for all vertices $w$, there exist some constants $c^w_i$ such that for all $k \geq 1$, $$\xi_{2k}(w) = \sum_{i = 1}^m c^w_i \lambda_i^{2k} = \sum_{i = 1}^m c^w_i (\lambda_i^2)^k = \sum_{i = 1}^n c^w_i (\lambda_i^2)^k$$ where in the last sum the $\lambda_i^2$ are all distinct positive constants (by combining like terms in the sum with $m$ terms to get a sum with $n$ terms). Then $$\sum_{w\in G} f(w)\xi_{2k}(w) = \sum_{w\in G} f(w)\sum_{i = 1}^n c^w_i (\lambda_i^2)^k.$$ Thus, there exist constants $c_f^1, \ldots c_f^n$ such that $\sum_{w\in G} f(w)\xi_{2k}(w) = \sum_{i=1}^n c_f^i (\lambda_i^2)^k.$ Let $\lambda_f^2 = \max_{i, c_f^i \neq 0} \lambda_i^2$ (this is well-defined since if it wasn't well defined we would have $\sum_{w\in G} f(w)\xi_{2k}(w) = 0$ for all $k \geq 1$). Let $c_f$ be the constant corresponding to this $\lambda_f^2$. Then $$\frac{\sum_{w\in G} f(w)\xi_{2k}(w)}{c_f\cdot \lambda_f^{2k}} = \frac{\sum_{i=1}^n c_f^i (\lambda_i^2)^k}{c_f\cdot \lambda_f^{2k}} = 1 + O(c^{2k}),$$ where $0 < c < 1$. Thus we have that $$\sum_{w\in G} f(w)\xi_{2k}(w) \sim c_f \cdot \lambda_f^{2k}.$$ 
\end{proof}

We now have all the pieces we need to show that $W_{2k}$ is approximately some exponential. The following theorem finishes off the proof.

\begin{theorem} \label{thm: Simba the sim-an-exponential when W=0 fairy}
For any Guvab $\mathcal{G}$ for which $W=0$ and $\{W_{2k}\}$ is not eventually constant, we have that there exists some $0 < \lambda_{\emph{\even}} < 1$ and some $c > 0$, such that $W_{2k} \sim c\cdot \lambda_{\emph{\even}}^{2k}$. 
\end{theorem}

\begin{proof}
By Lemma \ref{lem: finite function set}, there exists some set $F = \{f_1,\ldots f_n\}$ of 1-Lipschitz functions $f_i: V(G) \to \mathbb{R}$ such that for all $k$ there exists $f \in F$ such that $\displaystyle W_{2k} = \sum_{w\in G} f(w)\xi_{2k}(w)$. Furthermore, by Lemma \ref{lem: 1-lip are exponential}, for each of these $f \in F$ there exists some $\lambda_f$ and some positive constant $c_f$ such that $\sum_{w\in G} f(w)\xi_{2k}(w) \sim c_f \cdot \lambda_f^{2k}$, unless there exists some $N$ such that for all $k > N$ we have that $\sum_{w\in G} f(w)\xi_{2k}(w) = 0$. If for all $f \in F$, there exists some $N$ such that for all $k > N$ we have $\sum_{w\in G} f(w)\xi_{2k}(w) = 0$, then we have that $\{W_{2k}\}$ is eventually constant at 0. Otherwise, let $\Tilde{F}$ be the set of functions $f$ for which $\lambda_f$ is well-defined. Then let $\lambda_{\even}$ be $\max_{f \in \Tilde{F}} \lambda_f$. Let $F' \subset \Tilde{F}$ be the set of $f$ such that $\lambda_f = \lambda_{\even}$, and let $c$ be $\max_{f \in F'} c_f$. Finally, let $\mathcal{F} \subset \Tilde{F}$ be the set of $f \in \Tilde{F}$ such that $\lambda_f = \lambda_{\even}$ and $c_f = c$. Then for all $f \in F$ such that $f \notin \mathcal{F}$, there exists some $N$ such that for all $k \geq N$ we have $$\sum_{w\in G} f(w)\xi_{2k}(w) < \max_{f \in \mathcal{F}} \sum_{w\in G} f(w)\xi_{2k}(w) \leq W_{2k}.$$ Thus, since $W_{2k}$ must be the output of some 1-Lipschitz function, there exists some $N$ such that for all $k \geq N$ we have $W_{2k} = \sum_{w\in G} f(w)\xi_{2k}(w)$ for some $f \in \mathcal{F}$, since it cannot be the output of any 1-Lipschitz function $f \notin \mathcal{F}$. However, for all $f \in \mathcal{F}$, we have that $\sum_{w\in G} f(w)\xi_{2k}(w) \sim c\cdot \lambda_{\even}^{2k}$. Thus for all $k \geq N$, we have that $W_{2k} \sim c \cdot \lambda_{\even}^{2k}$. Hence $W_{2k} \sim c \cdot \lambda_{\even}^{2k}$.
\end{proof}

\begin{remark} \label{rem: odd simba}
Analogously, for any Guvab $\mathcal{G}$ for which $W=0$ and $\{W_{2k+1}\}$ is not eventually constant, we have that there exists some $0 < \lambda_{\emph{\odd}} < 1$ and some $c > 0$, such that $W_{2k+1} \sim c\cdot \lambda_{\emph{\odd}}^{2k+1}$. 
\end{remark}

We now seek to explicitly characterize all the cases where $W = 0$ and $W_k$ is eventually constant. We start by understanding why we only need to consider the first few terms of $\{W_k\}$ to characterize all of these cases.

\begin{lemma} \label{lem: Mustard the Must-be-constant-after-1-step fairy} When $\displaystyle\lim_{k\to\infty}W_k = 0$, if there exists some $N \geq 0$ such that $\{W(\mu_k,\nu_k)\}_{k\geq N}$ is a constant sequence, then $\displaystyle\{W(\mu_k,\nu_k)\}_{k\geq 1}$ is also a constant sequence.\end{lemma} 

\begin{proof} By Lemma \ref{lem: eileen the eigval sum fairy}, if we let the distinct eigenvalues of the transition matrices be $\lambda_1,\ldots, \lambda_n$, then for any vertex $w$ and for any $k \geq 1$ we can write $(\mu_k - \nu_k)_w = \sum_{i=1}^n c^w_i\lambda_i^k$ for some constants $c^w_1,\ldots, c^w_n$. Note that if $\{W(\mu_k,\nu_k)\}_{k\geq N}$ is a constant sequence, $0 = \lim_{k\to\infty}W_k = W(\mu_k,\nu_k)$ for all $k \geq N$. Thus for any vertex $w$, we will have $\sum_{i=1}^n c^w_i\lambda_i^k = 0$ for all $k \geq N$. 

Suppose that for some $i$, we have that $c^w_i$ and $\lambda_i$ are nonzero. Then let $\Lambda$ be the set of all $\lambda_i$ for which $c^w_i$ and $\lambda_i$ are nonzero. Then let $\lambda_m = \max_{\lambda \in \Lambda} |\lambda|$. If there is only one $\lambda_i \in \Lambda$ such that $|\lambda_i| = \lambda_m$, then for some $N$, for all $k > N$ we will have that $|c^w_i\lambda_i^k| > \sum_{j\neq i} |c^w_j\lambda_j^k|$ so the left-hand-side term will dominate and $\sum_{i=1}^n c^w_i\lambda_i^k$ will be nonzero. Then $0 \neq \lim_{k\to\infty}W_k$. If there is more than one $\lambda \in \Lambda$ such that $|\lambda| = \lambda_m$, then those two $\lambda$s will be $\lambda_m$ and $\lambda_{m'} = -\lambda_m$, since those are the only two numbers with absolute value $\lambda_m$. We know that $c^w_m\lambda_m^k$ will stay the same sign regardless of $k$, while $c^w_{m'}\lambda_{m'}^k$ will switch sign with parity. Thus, for one of the parities, $c^w_m\lambda_m^k$ and $c^w_{m'}\lambda_{m'}^k$ will have the same sign. Thus, for some $N$, either for all even $k > N$ or for all odd $k > N$, we will have that $|c^w_m\lambda_m^k + c^w_{m'}\lambda_{m'}^k| > \sum_{j\neq m,m'} |c^w_j\lambda_j^k|$, so the left-hand-side term will dominate and $\sum_{i=1}^n c^w_i\lambda_i^k$ will be nonzero. Then $0 \neq \lim_{k\to\infty}W_k$. Thus, we must have for all $1 \leq i \leq n$ that either $c^w_i$ or $\lambda_i$ is 0. 

This means that for all $1 \leq i \leq n$, either $c_i^w$ or $\lambda_i$ is 0. Thus, for all $k \geq 1$, we have that $c_i^w\lambda_i^k = 0$. Therefore $(\mu_k - \nu_k)_w = \sum_{i=1}^n c^w_i\lambda_i^k = 0$, so $\mu_k-\nu_k$ will be 0 at all vertices, so $W(\mu_k, \nu_k) = 0$ for all $k \geq 1$.  
\end{proof}

With this lemma established, we proceed to characterize all the cases when the Wasserstein distance is eventually constant in the case where $W = 0$. 

\begin{theorem} \label{thm: Constance the constant distance when W = 0 fairy}
When $\lim_{k\to\infty}W_k = 0$, we have that $W_k$ is eventually constant if and only if one of the following holds:
\begin{itemize}
    \item $\alpha = \beta = 0$ and $N(u)=N(v)$,
    \item $\displaystyle \alpha = \beta = \frac{1}{\deg u + 1}$, the edge $\{u,v\}\in E(G)$, and if the edge $\{u,v\}$ were removed from $E(G)$ then $u,v$ would have $N(u)=N(v)$,
    \item $\alpha = \beta$ and $u = v$.
\end{itemize}
\end{theorem}

\begin{proof}
We know by Lemma \ref{lem: Mustard the Must-be-constant-after-1-step fairy} that if $\lim_{k\to\infty}W_k = 0$ and $W_k$ is eventually always 0, then $\mu_1 = \nu_1$ and $\mu_2 = \nu_2$. Let $\mu_1 = \nu_1$ be $\phi$. Recall that $P_\alpha$ is the transition matrix for $X$ and $P_\beta$ is the transition matrix for $Y$. Further recall that $P_\alpha = \alpha I + (1-\alpha) P$ and $P_\beta = \beta I + (1-\beta) P$. Then we have $\phi P_\alpha = \phi P_\beta$, so $\phi(\alpha I + (1-\alpha) P) = \phi(\beta I + (1-\beta) P)$. Then $\phi((\beta - \alpha)I + (\alpha - \beta)P) = 0$. 

If $\alpha \neq \beta$ then dividing out by $(\alpha-\beta)$, we get $\phi P = \phi$. If such a $\phi$ exists, it must be the stationary distribution $\pi$. Then $\phi = \mu_0 P = \mathbbm{1}_u P = \pi$. However, $(\mathbbm{1}_u P)_u = P_{u,u} = 0$ by definition of $P$, and $\pi_u = \frac{\deg(u)}{2|E(G)|} > 0$. Thus, we cannot have $\mathbbm{1}_u P = \pi$, so we cannot have $\alpha \neq \beta$.

This means we have that $\alpha = \beta$. 

We also know that, given that $\lim_{k\to\infty}W_k = 0$, if $\mu_1 = \nu_1$ and $\alpha = \beta$, then $\mu_k = \mu_1(P_\alpha)^{k-1} = \nu_1(P_\alpha)^{k-1} = \nu_k$ for all $k \geq 1$ so $W_k$ is eventually always 0. 

It therefore suffices to characterize the cases where $\lim_{k\to\infty}W_k = 0$ and $\mu_1 = \nu_1$ and $\alpha = \beta$. We first note that if $u = v$ and $\alpha = \beta$, we are done. Otherwise, we assume that $\alpha = \beta$ and casework on the values of $\alpha$ to determine which cases yield $\lim_{k\to\infty}W_k = 0$ and $\mu_1 = \nu_1$. 

If $\alpha = 0$, we need that $u$ and $v$ have the same neighbor set, since if $u$ had some neighbor $n$ that was not adjacent to $v$ then $\mu_1$ would have nonzero mass at $n$ and $\nu_1$ would not. We will also show that this is a sufficient condition. If $u$ and $v$ have the same neighbor set then $\deg(u) = \deg(v)$. For each neighbor $n$ of $u$ and $v$, we have that $\mu_1(n) = \frac{1-\alpha}{\deg(u)} = \frac{1-\beta}{\deg(v)} = \nu_1(n)$ and for all other vertices $w$, we have that $\mu_1(w)= \nu_1(w) = 0$. Thus $\mu_1 = \nu_1$. We also know that $\lim_{k\to\infty}W_k = 0$ by Theorem \ref{thm: 0-convergence} since $\alpha = \beta = 0$ and for any neighbor $n$ of $u$, the path $u \to n \to v$ has an even number of steps. 

If $0 < \alpha < 1$, we first note that we need $u$ and $v$ to be adjacent, since $\mu_1(u) = \alpha > 0$ and $\nu_1(u) = 0$ if $u$ and $v$ are not adjacent. When $u$ and $v$ are adjacent we have that $\nu_1(u) = \frac{1-\alpha}{\deg(u)}$, so since $\nu_1(u) = \mu_1(u)$ we have $\frac{1-\alpha}{\deg(u)} = \alpha$, which yields $\alpha = \frac{1}{\deg u + 1}$. We also note that, similarly to before, aside from the edge $\{u,v\}$ we have that $u$ and $v$ need to have the same set of neighbors because if there was some vertex $n\neq u, v$ such that $n\sim u$ and $n\not \sim v$, then $\mu_1$ would have nonzero mass at $n$ and $\nu_1$ would not. We will finish by showing that if $\alpha,\beta,u,v$ satisfy these conditions, then $\mu_1 = \nu_1$ and $\lim_{k\to\infty}W_k = 0$. 

Suppose that the conditions are satisfied. We know that $\deg(u) = \deg(v)$, so $\mu_1(u) = \alpha = \frac{1-\alpha}{\deg(u)} = \nu_1(u)$ and similarly $\mu_1(v) = \nu_1(v)$. We also know that for all $n\neq u,v$ such that $n\sim u$ and $n\sim v$, we have that $\mu_1(n) = \frac{1-\alpha}{\deg(u)} = \frac{1-\beta}{\deg(v)} = \nu_1(n)$ and for all other vertices $w$, we have that $\mu_1(w)= \nu_1(w) = 0$. Thus $\mu_1 = \nu_1$. Also, $\lim_{k\to\infty}W_k = 0$ by Theorem \ref{thm: 0-convergence} since $0 < \alpha \leq \beta < 1$.
\end{proof}

\section{Convergence when $\beta = 1$}
We next consider the case of Guvabs where $\beta = 1$. Similarly to the $W = 0$ case, we show that the rate of convergence is exponential unless the distance is eventually constant. Furthermore, when the Wasserstein distance is eventually constant, it is constant after exactly 1 step. 

We first show that the rate of convergence is exponential unless the distance is eventually constant.
\begin{lemma} \label{lem: betty the beta=1 theorem}
Consider a Guvab where $\beta = 1$. Either $\{W_{2k}\}$ is eventually constant, or for some $c_e$ and some $\lambda_e$, we have that $|W_{2k} - \lim_{k\to\infty} W_{2k}| \sim c_e\cdot \lambda_e^{2k}$. Also, either $\{W_{2k+1}\}$ is eventually constant, or for some $c_o$ and some $\lambda_o$, we have that $|W_{2k+1} - \lim_{k\to\infty} W_{2k+1}| \sim c_o\cdot \lambda_o^{2k+1}$.
\end{lemma}

\begin{proof}
When $\beta = 1$, we know that $W_k = \sum_{w \in G} \mu_k(w)\dist(w,v) = \sum c_i \cdot \lambda_i^k$ for some constants $c_i$ and $\lambda_i$. Using the same reasoning as in the proof of Lemma \ref{lem: 1-lip are exponential}, we know that (unless $\sum c_i \cdot \lambda_i^{2k}$ is eventually constant) $\sum c_i \cdot \lambda_i^{2k} \sim c_e\cdot \lambda_e^{2k}$ for some $c_e, \lambda_e$. We also know that (unless $\sum c_i \cdot \lambda_i^{2k+1}$ is eventually constant) $\sum c_i \cdot \lambda_i^{2k+1} \sim c_o\cdot \lambda_o^{2k+1}$ for some $c_o, \lambda_o$. Thus, we attain the desired result.
\end{proof}

We now show that if the distance is eventually constant, it is constant after 1 step.
\begin{lemma} \label{lem: when beta is 1 we must be constant after 1 step} 
When $\beta = 1$, if there exists some $N \geq 0$ such that $\{W(\mu_n,\nu_n)\}_{n\geq N}$ is a constant sequence, then $\{W(\mu_n,\nu_n)\}_{n\geq 1}$ is also a constant sequence.\end{lemma}

\begin{proof}
When $\beta = 1$, we have that $W_k$ is $\sum_{w \in G} \mu_k(w)\dist(w,v) = \sum c_i \cdot \lambda_i^k$ for some constants $c_i$ and $\lambda_i$. Thus, for similar reasons as in the proof of Lemma \ref{lem: Mustard the Must-be-constant-after-1-step fairy}, all the $c_i$ for $\lambda_i \neq 0,1$ are 0 so $\{W(\mu_n,\nu_n)\}_{n\geq 1}$ is constant.
\end{proof}

When $W = 0$, using a lemma similar to Lemma \ref{lem: when beta is 1 we must be constant after 1 step} we were able to explicitly characterize exactly when $W_k$ was eventually constant. Lemma \ref{lem: when beta is 1 we must be constant after 1 step} provides an important step towards making a similar characterization when $\beta=1$. To exemplify how a characterization could be made when $\beta=1$, we provide a family of examples of Guvabs where $W_k$ is eventually constant. 

\begin{definition}
We define a \textbf{Gluvab} $\mathcal{J}$ to be a Guvab that satisfies all of the following conditions:
\begin{itemize}
    \item $\beta = 1$,
    \item $2\dist(u,v) = \max_{w\in G} \dist(w,v)$,
    \item if $\dist(x,v) = \max_{w\in G} \dist(w,v)$, then for all $n \sim x$ we have that $\dist(n,v) < \dist(x,v)$,
    \item if $0 < \dist(x,v) < \max_{w\in G} \dist(w,v)$, then for exactly half of the neighbors $n\sim x$ we have that $\dist(n,v) < \dist(x,v)$, and for exactly the other half we have that $\dist(n,v) > \dist(x,v)$.
\end{itemize}
\end{definition}
\begin{example}
Consider a Guvab with $G=P_3$ (where $P_3$ is the path graph with $3$ vertices), $v$ is the vertex of $P_3$ with degree 2, $u$ is either of the other two vertices, $\alpha=\frac{1}{3}$, and $\beta=1$. One can check that this Guvab is a Gluvab.
\end{example}
\begin{lemma}\label{lem: Garry the Gluvab Fairy}
Any Gluvab $\mathcal{J}$ satisfies $W_0 = W_1 = \cdots$.
\end{lemma}

\begin{proof}
We aim to prove this lemma by essentially reducing each Gluvab to a random walk on a path graph. In particular, each vertex $m_i$ in the path corresponds to the set of vertices $\{w\in G : d(w,v)= i\}$ at a given distance $i$ from $v$. After this, the desired result follows without much difficulty.

Construct the Markov chain $M$ that is simply a random walk with laziness $\alpha$ on a path of length $2\dist(u,v)$ with vertices $m_0, m_1, \ldots, m_{\dist(u,v)}, \ldots, m_{2\dist(u,v)}$. We let the starting point of this Markov chain be $m_{\dist(u,v)}$. It suffices to show that for all $i$, we have that $\displaystyle \sum_{w \in G, \, \dist(w,v) = i} \mu_k(w) = M_k(m_i)$, because that would mean that the distribution is always symmetric about $u$ so $\mu_k$ always has the same average distance $\dist(u,v)$.

We will show by induction on $k$ that for all $i$, $$ \sum_{w \in G, \, \dist(w,v) = i} \mu_k(w) = M_k(m_i).$$

Base case: At $k = 0$, we have that $\mu_k$ is only nonzero at $u$ and that $M_k$ is only nonzero at $m_{\dist(u,v)}$, so the claim holds.

Inductive step: We suppose that this claim holds for $k$. We will show that it holds for $k+1$. We know the following facts about $M$:
\begin{itemize}
    \item $M_{k+1}(m_0) = \frac{1-\alpha}{2}M_k(m_1) + \alpha M_k(m_0)$,
    \item $M_{k+1}(m_{2\dist(u,v)}) = \frac{1-\alpha}{2}M_k(m_{2\dist(u,v) - 1}) + \alpha M_k(m_{2\dist(u,v)})$,
    \item $M_{k+1}(m_1) = \frac{1-\alpha}{2}M_k(m_2) + \alpha M_k(m_1) + (1-\alpha) M_k(m_0)$,
    \item $M_{k+1}(m_{2\dist(u,v) - 1}) = \frac{1-\alpha}{2}M_k(m_{2\dist(u,v) - 2}) + \alpha M_k(m_{2\dist(u,v) - 1}) + (1-\alpha) M_k(m_{2\dist(u,v)})$,
    \item for $1 < i < 2\dist(u,v)-1$, we have that $M_{k+1}(m_i) = \alpha M_k(m_i) + \frac{1-\alpha}{2}(M_k(m_{i-1}) + M_k(m_{i+1})).$
\end{itemize}
We now examine $\mu_{k+1}$, and in particular the amount of mass of $\mu_{k+1}$ at each level. We let $S_k(i)$ denote the mass of $\mu_k$ at the $i$th level; in other words, $$S_k(i) = \displaystyle \sum_{w \in G, \, d(w,v)=i} \mu_k(w).$$ For all $i$, we can calculate $S_{k+1}(i)$ by considering the $i$th level and considering how much mass from each level from $S_k$ goes to the $i$th level. This is possible because all vertices at the same level will have indistinguishable behavior with respect to their contribution to the $i$th level. By calculating the contribution of each different level to the $i$th level, we can check that 
\begin{itemize}
    \item $S_{k+1}(0) = \frac{1-\alpha}{2}S_k(1) + \alpha S_k(0)$,
    \item $S_{k+1}(2\dist(u,v)) = \frac{1-\alpha}{2}S_k(2\dist(u,v) - 1) + \alpha S_k(2\dist(u,v))$,
    \item $S_{k+1}(1) = \frac{1-\alpha}{2}S_k(2) + \alpha S_k(1) + (1-\alpha) S_k(0)$,
    \item $S_{k+1}(2\dist(u,v) - 1) = \frac{1-\alpha}{2}S_k(2\dist(u,v) - 2) + \alpha S_k(2\dist(u,v) - 1) + (1-\alpha) S_k(2\dist(u,v))$,
    \item for $1 < i < 2\dist(u,v)-1$, we have that $S_{k+1}(i) = \alpha S_k(i) + \frac{1-\alpha}{2}(S_k(i-1) + S_k(i+1)).$
\end{itemize}
This lines up exactly with our characterization of $M_{k+1}$, so for all $i$ we have $$\displaystyle \sum_{w \in G, \dist(w,v) = i} \mu_{k+1}(w) = M_{k+1}(m_i).$$
\end{proof}

\section{Main Convergence Theorems}
Since we have shown that all Guvabs have $W = 1$ or $W = \frac{1}{2}$ or $W = 0$ or $\beta = 1$, and we have some understanding of the rate of convergence of the Wasserstein distance in each of these cases, we make some general statements about convergence that apply to all Guvabs. The following theorems sum up the general convergence results obtained from considering the each of the cases $W = 1$, $W = \frac{1}{2}$, $W = 0$ and $\beta = 1$ in the previous sections. 

The first theorem states that the rate of convergence of $\{W_{2k}\}$ and $\{W_{2k+1}\}$ is exponential unless it is eventually constant.
\begin{theorem} \label{thm: Guvab Convergence Theorem}
For any Guvab, we have that
\begin{itemize}
    \item either $\{W_{2k}\}$ is eventually constant, or there exists a constant $\lambda_{\emph{\even}} \in (-1,1)$ and a positive constant $c_{\emph{\even}} > 0$ such that $|W_{2k} - \lim_{k\to\infty}W_{2k}| \sim c_{\emph{\even}} \cdot |\lambda_{\emph{\even}}|^{2k}$,
    \item either $\{W_{2k+1}\}$ is eventually constant, or there exists a constant $\lambda_{\emph{\odd}} \in (-1,1)$ and a positive constant $c_{\emph{\odd}} > 0$ such that $|W_{2k+1} - \lim_{k\to\infty}W_{2k+1}| \sim c_{\emph{\odd}} \cdot |\lambda_{\emph{\odd}}|^{2k+1}$
\end{itemize}
\end{theorem}
\begin{proof}
To begin, note that when $\beta<1$, we have that $W_k$ converges and $W\in\{1,\frac{1}{2},0\}$ by Corollary \ref{cor: four guvab nations}. Further, when $\beta=1$, Lemma \ref{lem: betty the beta=1 theorem} implies exactly that the desired result holds. Thus, it suffices to consider each of these cases $W=1$, $W=\frac{1}{2}$, and $W=0$ separately.

First, when $W=1,$ Theorem \ref{thm:winnie the when it's constant for W = 1 fairy} implies $\{W_k\}$ is eventually constant (and hence, we have the same for $\{W_{2k}\}$ and $\{W_{2k+1}\}$). This gives the desired result in the case $W=1$.

When $W=\frac{1}{2},$ Theorem \ref{thm: washington the w=1/2 convergence theorem} implies that either $\beta=\frac{1}{2}$ and $\{W_k\}$ is eventually constant or else $|W_{k} - \lim_{k\to\infty}W_{k}| \sim 0.5 \cdot |1-2\beta|^{2k}$ (and hence, we have the same for $\{W_{2k}\}$ and $\{W_{2k+1}\}$). This gives the desired result for $W=\frac{1}{2}.$

Finally, we note that when $W=0$, Theorem \ref{thm: Simba the sim-an-exponential when W=0 fairy} (and Remark \ref{rem: odd simba}) gives exactly the desired result. Thus, having checked each case, we conclude the proof.
\end{proof}

The second theorem provides a characterization of when $\{W_k\}$ is eventually constant when $\beta < 1$.
\begin{theorem} \label{thm: Characterization of Constancy}
When $\beta < 1$, we have that $\{W_k\}$ is eventually constant if and only if one of the following holds:
\begin{itemize}
    \item $\alpha = \beta = 0$, the graph $G$ is bipartite, and $\dist(u,v)$ is odd,
    \item $\alpha = 0$ and $\beta = \frac{1}{2}$, and $G$ is bipartite,
    \item $\alpha = \beta = 0$ and $N(u)=N(v)$,
    \item $\displaystyle \alpha = \beta = \frac{1}{\deg u + 1}$, the edge $\{u,v\}\in E(G)$, and if the edge $\{u,v\}$ were removed from $E(G)$ then $u,v$ would have $N(u)=N(v)$,
    \item $\alpha = \beta$ and $u = v$.
\end{itemize}
\end{theorem}
\begin{proof}
To begin, note that when $\beta<1$, we have that $W_k$ converges and $W\in\{1,\frac{1}{2},0\}$ by Corollary \ref{cor: four guvab nations}. Thus, it suffices to consider each of these cases where $W=1$, $W=\frac{1}{2}$ and $W=0$ separately.

First, we look at the case where $W=1$. Note that, in this case, Theorem \ref{thm:winnie the when it's constant for W = 1 fairy} implies that $\{W_k\}$ is always eventually constant. Further, by Theorem \ref{thm: convergence values}, we see this case is equivalent to $\alpha=\beta=0$ and $W \neq 0$. Further, by Theorem \ref{thm: 0-convergence}, this case occurs exactly when $\alpha = \beta = 0$, the graph $G$ is bipartite, and $\dist(u,v)$ is odd (i.e., the first item of the theorem statement).

Next, when $W=\frac{1}{2}$, Corollary \ref{cor: wendy the when it's constant for W = 1/2 fairy} implies $\{W_k\}$ is eventually constant exactly when $\beta=\frac{1}{2}$. By Theorem \ref{thm: convergence values},  this case occurs exactly when $\alpha = 0$ and $\beta = \frac{1}{2}$, and $G$ is bipartite (i.e., the second item of the theorem statement).

Finally, when $W=0$, we see that Theorem \ref{thm: Constance the constant distance when W = 0 fairy} implies $\{W_k\}$ is eventually constant exactly when one of the following holds:
\begin{itemize}
    \item $\alpha = \beta = 0$ and $N(u)=N(v)$,
    \item $\displaystyle \alpha = \beta = \frac{1}{\deg u + 1}$, the edge $\{u,v\}\in E(G)$, and if the edge $\{u,v\}$ were removed from $E(G)$ then $u,v$ would have $N(u)=N(v)$,
    \item $\alpha = \beta$ and $u = v$.
\end{itemize}
Note that, each of these cases is indeed a case where $W=0$ by Theorem \ref{thm: 0-convergence}, so this case is equivalent to the final three items of the theorem statement.

Thus, considering each of these cases together, we obtain the desired result.
\end{proof}


\section{Open Problems}
The theorems presented in this paper open up several new questions and directions for further research, which the reader is invited to consider. Specifically, given Theorem~\ref{thm: Guvab Convergence Theorem}, the remaining questions regarding the behavior of Guvabs can be broken into three main categories: 1) determining when $\{W_{2k}\}$ and $\{W_{2k+1}\}$ are eventually constant, 2) in cases $\{W_{2k}\}$ and $\{W_{2k+1}\}$ are eventually constant, determining how long they take to become constant, and 3) determining $c$ and $\lambda$ when $\{W_{2k}\}$ and $\{W_{2k+1}\}$ are not eventually constant. In this section, we break down what we have shown and what is left to be done regarding each of these questions.

By Theorem \ref{thm: Characterization of Constancy}, we have characterized the cases where $\{W_k\}$ is constant in all cases where $\beta < 1$. Furthermore, in the cases of $W = 1$ and $W = \frac{1}{2}$, we know that $\{W_{2k}\}$ is eventually constant if and only if $\{W_k\}$ is eventually constant, and similarly $\{W_{2k+1}\}$ is eventually constant if and only if $\{W_k\}$ is eventually constant. In the case of $W = 0$, it remains to characterize the cases where either $\{W_{2k}\}$ or $\{W_{2k+1}\}$ individually are eventually constant, but $\{W_k\}$ is not. Further, in the $\beta = 1$ case we lack a complete characterization of when $\{W_k\}$ is eventually constant.

Question $2)$ remains largely unanswered and is a promising direction for future work. The progress so far in this paper is restricted to fairly weak upper and lower bounds when $W = 1$, and characterizations of when $\{W_k\}$ is eventually constant when $W = 0$ and $\beta = 1$. One interesting problem is that of tighter bounds for the case where $W = 1$, and similar bounds for the case when $W = \frac{1}{2}$ and $W$ is eventually constant. Also, depending on the answers to Question 1, there may be Guvabs where only one of $\{W_{2k}\}$ and $\{W_{2k+1}\}$ is eventually constant. If we find a specific Guvab that satisfies these criteria, it will be interesting to determine how long this Guvab takes to have either $\{W_{2k}\}$ or $\{W_{2k+1}\}$ be eventually constant.

Answering question $3)$ will require specific knowledge of eigenvectors and eigenvalues. In full generality, this is difficult, so a potential direction for future work would be addressing it in specific examples.

\section{Acknowledgements}

We would like to thank our mentor, Pakawut Jiradilok, for providing us with important knowledge, guidance, and assistance throughout our project. We would also like to thank Supanat Kamtue for the problem idea and helpful thoughts and guidance. Finally, we would like to thank the PRIMES-USA program for making this project possible. 
\newpage
\bibliographystyle{alpha}
\bibliography{wryan_the_wryteup_fairy}
\end{document}